\documentclass[a4paper,12pt,oneside]{article}

\usepackage{amsmath}
\usepackage{amssymb}
\usepackage{amsthm}
\usepackage{t1enc}
\usepackage[english]{babel}
\usepackage{verbatim} 
\usepackage{graphicx}
\usepackage[dvips]{epsfig}
\usepackage{psfig}
\usepackage{url}
\usepackage{tikz}
\usetikzlibrary{intersections, shapes}
\usepackage{subcaption}
\usepackage{booktabs,multirow}
\usepackage{bigdelim}
\usepackage[hidelinks]{hyperref}
\usepackage{array}
\usepackage{graphicx,bm,xcolor}
\usetikzlibrary{arrows}

\theoremstyle{plain}
\newtheorem{theorem}{Theorem}[section]

\newtheorem{lemma}[theorem]{Lemma}

\theoremstyle{definition}
\newtheorem{assumption}{Assumption}

\theoremstyle{remark}
\newtheorem*{notation}{Notation}

\title{Convergence theory for IETI-DP solvers for discontinuous Galerkin
Isogeometric Analysis \\ that is
explicit in $h$ and $p$ }

\author{Rainer Schneckenleitner\footnote{\texttt{schneckenleitner@numa.uni-linz.ac.at},
Institute of Computational Mathematics, Johannes
Kepler University Linz, Austria}\; and Stefan Takacs\footnote{\texttt{stefan.takacs@ricam.oeaw.ac.at},
Johann Radon Institute Institute for Computational and Applied Mathematics, Austrian Academy of Sciences, Linz, Austria}}
\date{}

\begin{document}
\maketitle

\selectlanguage{english}
\begin{abstract}
In this paper, we develop a convergence theory for Dual-Primal Isogeometric Tearing
and Interconnecting (IETI-DP) solvers for isogeometric multi-patch discretizations of the
Poisson problem, where the patches are coupled using
discontinuous Galerkin. The presented theory provides condition
number bounds that are explicit in the grid sizes $h$ and in the spline
degrees $p$. We give an analysis that holds for various choices for the primal degrees
of freedom: vertex values, edge averages, and a combination of both.
If only the vertex values or
both vertex values and edge averages are taken as primal degrees of freedom,
the condition number bound is the same as for the conforming case. If only the
edge averages are taken, both the convergence theory and the experiments show
that the condition number of the preconditioned system grows with the ratio of the
grid sizes on neighboring patches.
\end{abstract}

\section{Introduction}
\label{sec:1}

Isogeometric Analysis (IgA), see~\cite{HughesCottrellBazilevs:2005,Cottrell:Hughes:Bazilevs}, is
an approach for discretizing partial differential equations (PDEs) that
has been developed in order to improve the compatibility between computer aided
design (CAD) and simulation in comparison to the standard finite element method (FEM).
The geometry function, that is used in the CAD system to parameterize
the computational domain, is also used for the simulation.
These geometry functions are usually spanned
by B-splines or non-uniform rational B-splines (NURBS). Following the
principle of IgA, we also use such functions for the discretization
of the PDE. Since only simple domains can be represented by just one geometry
function (single-patch case), the overall computational domain is usually decomposed into multiple patches,
each of which is parameterized using its own geometry function (multi-patch IgA).
We focus on non-overlapping patches.

The patches can be coupled either in a conforming way or by means of
discontinuous Galerkin methods. For a conforming discretization both the
geometry function and the discretization have to agree on the interfaces
between the patches. One promising alternative to overcome these restrictions
are discontinuous
Galerkin approaches, cf.~\cite{Riviere:2008,Arnold:Brezzi:Cockburn:Marini:2002}, particularly
the symmetric interior penalty discontinuous Galerkin (SIPG) method,
cf.~\cite{Arnold:1982}. 
The idea of applying this technique to couple patches in IgA, has been previously discussed in~\cite{LangerMantzaflaris:2015,LangerToulopoulos:2015,Takacs:2019b}.

After the discretization of the PDE, we obtain a large-scale linear system and
we are interested in fast iterative solvers for such systems.
Since we are in a multi-patch
framework, domain decomposition solvers are a canonical choice. One of the most
popular domain decomposition solvers for large-scale systems of finite element equations in standard FEM
is the Finite Element Tearing and Interconnecting method (FETI), originally
proposed in~\cite{FarhatRoux:1991a}. Since the invention of FETI, various FETI-type methods have been developed, cf.~\cite{Pechstein:2013a, ToselliWidlund:2005a, KorneevLanger:2015}. In \cite{KleissPechsteinJuttlerTomar:2012}, it was proposed
to use a FETI-type method, namely the Dual-Primal FETI method (FETI-DP),
in the context of IgA. This method was called Dual-Primal Isogeometric Tearing and Interconnecting (IETI-DP) method. Later, this approach has been further analyzed, particularly
in~\cite{HoferLanger:2019a} and more recently in~\cite{SchneckenleitnerTakacs:2019}.
In the latter paper, the authors of the paper at hand have analyzed the
dependence of the condition number of the preconditioned system on the spline
degree $p$.

The extension of FETI methods or IETI methods to dG discretizations is not
straight-forward. We follow the approaches that have been proposed
in~\cite{DryjaGalvis:2013} and adapted for IETI in~\cite{Hofer:2016a}, particularly
the idea of using artificial interfaces. Up to the knowledge of the authors,
so far, there is no convergence analysis that covers the dependence of the
condition number of the preconditioned system on the spline degree $p$. Moreover,
we are not aware of convergence analysis for IETI-DP for dG discretizations
that covers the case that only edge averages are taken as primal degrees of
freedom. Such approaches might be of interest if the overall computational domain
in decomposed into patches in a way that allows T-junctions or in the case of moving patches like in the case of rotating electrical machines.

For conforming finite element discretizations, condition number bounds that are explicit in the polynomial degree $p$ have been worked out previously
for FETI-DP type methods, cf.~\cite{PAVARINO:2006, KlawonnPavarino:2008},
and other Schwarz type, cf.~\cite{SchoeberlMelenkPechsteinZaglmayr:2007, GuoCao:1997}, and iterative substructuring methods, cf,~\cite{Ainsworth:1996, PavarinoWidlund:1996}.

For dG discretizations, there are only a few results concerning a
$p$-analysis for domain decomposition approaches, like
\cite{AntoniettiAyuso:2007,SchoeberlLehrenfeld:2013} for two-level Schwarz
type approaches or \cite{DryjaGalvis:2007,CanutoPavarinoPieri:2014},
where the estimates for BDDC and FETI-DP type methods for spectral
element methods and $hp$-FEM are given.  
The publication \cite{AntoniettiHouston:2011} shows a bound for general non-overlapping Schwarz preconditioners for dG $hp$-FEM that is in the order of $p^2$. A polylogarithmic bound for a BDDC preconditioner for a hybridizable discontinuous Galerkin discretization has been developed in \cite{DiosadyDarmofal:2013}.

In this paper, we consider the Poisson problem on planar domains.
We consider a IETI-DP solver with a scaled Dirichlet preconditioner. The proof follows the abstract framework
from~\cite{MandelDohrmannTezaur:2005a} and is a continuation of
the paper~\cite{SchneckenleitnerTakacs:2019}, where we have
analyzed IETI-DP methods for conforming discretizations. 
The presented theory covers three different choices for the primal degrees of
freedom: vertex values only (Alg.~A), edge averages only (Alg.~B), and the
combination of both (Alg.~C). We prove that the condition of the preconditioned IETI-DP solver is bounded by 
\[
	C\, p \; 
	\left(1+\log p+\max_{k=1,\ldots,K} \log\frac{H_{k}}{h_{k}}\right)^2,
\]
for Alg.~A and C and by
\[
	C\, \delta\,  p
	 \left(
	\max_{k=1,\ldots,K}\max_{\ell \in \mathcal N_{\Gamma}(k)}
	\frac{h_{k}}{h_{\ell}}
	\right)
	\left(1+\log p+\max_{k=1,\ldots,K} \log\frac{H_{k}}{h_{k}}\right)^2
\]
for Alg.~B, where $p$ is the spline degree, $h_k$ are the grid sizes and $H_k$ are the patch sizes, $\mathcal N_{\Gamma}(k)$ contains the indices of the patches that share an edge with the $k$-th patch, and $\delta\ge\delta^*$ is a suitably chosen penalty parameter. Both the constant $C$ and the optimal penalty parameter $\delta^*>0$ are independent of the grid sizes, the patch sizes, the spline degree and the smoothness of the splines. Hence, the theory covers all discretizations where the smoothness within the patches is between $C^0$ and $C^{p-1}$.

The structure of this paper is as follows. In Section \ref{sec:2}, we introduce our model problem and its discretization using SIPG. Then, in Section \ref{sec:3}, we present the IETI-DP solver. Section~\ref{sec:4} is devoted to the proof of
the condition number bounds. The numerical results are shown in
Section \ref{sec:5}. In Section~\ref{sec:6}, we summarize our findings
and give some further remarks. 

\section{The model problem}
\label{sec:2}

We consider the discretization of a homogeneous Poisson problem
using multi-patch Isogeometric Analysis,
where the coupling between the individual patches is realized using a
SIPG approach.
Since this paper extends the results of the previous paper
\cite{SchneckenleitnerTakacs:2019}, which covered conforming discretizations,
we aim to use the same notation as in the aforementioned paper.
To keep the paper self-contained, we briefly reintroduce the notation.

We consider an open, bounded and simply connected domain $\Omega \subset \mathbb{R}^2$ with
Lipschitz boundary $\partial \Omega$. We are interested in solving the homogeneous Poisson problem: find $u \in H^1_0(\Omega)$
such that 
\begin{align}
	\label{continousProb}
	\int_{\Omega}^{} \nabla u \cdot \nabla v \; \mathrm{d}x = \int_{\Omega}^{} fv \; \mathrm{d}x \qquad\text{for all}\qquad v \in H^1_0(\Omega),
\end{align}
where $f \in L_2(\Omega)$ is a given function.
Throughout the paper, we denote by $L_2(\Omega)$ and $H^s(\Omega)$, $s \in \mathbb{R}$, the usual Lebesgue and Sobolev spaces, respectively. $H^1_0(\Omega) \subset H^1(\Omega)$ is the subspace of functions
whose trace vanishes on $\partial \Omega$.
We equip these Hilbert spaces with the usual scalar products $(\cdot, \cdot)_{L_2(\Omega)}$ and $(\cdot, \cdot)_{H^1(\Omega)}:=(\nabla \cdot, \nabla \cdot)_{L_2(\Omega)}$, norms $\| \cdot \|_{L_2(\Omega)}$, $\| \cdot \|_{H^s(\Omega)}$, and seminorms $| \cdot |_{H^s(\Omega)}$.

The domain $\Omega$ is the composition of $K$ non-overlapping subdomains $\Omega^{(k)}$, i.e., 
\begin{align*}
	\overline{\Omega} = \bigcup_{k=1}^K \overline{\Omega^{(k)}} \quad \text{and}\quad
	\Omega^{(k)} \cap \Omega^{(\ell)} = \emptyset \quad\text{for all}\quad k \neq \ell,
\end{align*}
where $\overline{T}$ denotes the closure of the set $T$. We refer to the subdomains $\Omega^{(k)}$ as patches. We assume that every patch $\Omega^{(k)}$ is the image of a geometry function 
\begin{align}
	G_k:\widehat{\Omega}:=(0,1)^2 \rightarrow \Omega^{(k)}:=G_k(\widehat{\Omega}) \subset \mathbb{R}^2, 
\end{align}
that can be continuously extended to the closure of $\widehat\Omega$.
Although the geometry functions can be arbitrary functions,
in IgA commonly B-spline or NURBS functions are used. We assume that the
geometry function is sufficiently smooth such that the following assumption holds.
\begin{assumption}
	\label{ass:nabla}
	There is a constant $C_1>0$ such that 
	\begin{align*}
		\| \nabla G_k \|_{L_\infty(\widehat{\Omega})} \le C_1\, H_{k}
		\quad\text{and}\quad
		\| (\nabla G_k)^{-1} \|_{L_\infty(\widehat{\Omega})} \le C_1\, \frac{1}{H_{k}}
	\end{align*}
	for all $k=1,\ldots,K$, where $H_k > 0$ is the diameter of the patch $\Omega^{(k)}$. 
\end{assumption}

The next assumption guarantees that $\Omega$ does not have any T-junctions.
\begin{assumption}
	\label{ass:conforming}
	For any two patch indices $k\not=\ell$, the intersection
	$\overline{\Omega^{(k)}} \cap \overline{\Omega^{(\ell)}}$
	is either a common edge (including the neighboring vertices),
	a common vertex, or empty.
\end{assumption}
This assumption also ensures that the pre-images $\widehat \Gamma_D^{(k)}
:= G_k^{-1}(\partial\Omega\cap\partial\Omega^{(k)})$ of the (Dirichlet) boundary $\Gamma_D=\partial
\Omega$ consist of whole edges.

For any two neighboring patches $\Omega^{(k)}$ and $\Omega^{(\ell)}$, the common edge is denoted by $\Gamma^{(k,\ell)}=\Gamma^{(\ell,k)}$, and its pre-images by
$\widehat{\Gamma}^{(k,\ell)}:=G_k^{-1}(\Gamma^{(k,\ell)})$ and
$\widehat{\Gamma}^{(\ell,k)}:=G_{\ell}^{-1}(\Gamma^{(\ell,k)})$. We collect the indices of patches sharing an edge in a set:
\[
	\mathcal N_\Gamma(k):=
		\{ \ell \neq k \;:\; \Omega^{(k)}\mbox{ and }\Omega^{(\ell)}
		\mbox{ share at least one edge}\}.
\]
We do the same for two patches $\Omega^{(k)}$ and $\Omega^{(\ell)}$ sharing only a vertex. We denote that common vertex by $\textbf x^{(k,\ell)}=\textbf x^{(\ell,k)}$,
as well as its representations in the parameter domain by
$\widehat{\textbf x}^{(k,\ell)}:=G_k^{-1}(\textbf x^{(k,\ell)})$ and
$\widehat{\textbf x}^{(\ell,k)}:=G_{\ell}^{-1}(\textbf x^{(\ell,k)})$. The
indices of the patches that contain a certain corner are collected in the set
\[
		\mathcal{P}(\textbf x):= \{k\;:\; \textbf x\in \overline{\Omega^{(k)}} \}.
\]
Moreover, we require that the number of neighbors of a patch is uniformly bounded.
\begin{assumption}
	\label{ass:neighbors}
	There is a constant $C_2>0$ such that
	$
		| \mathcal{P}(\textbf x)| \le C_2
	$ 
	holds for every corner $\textbf x$.
\end{assumption}

After the introduction of the computational domain, we establish the isogeometric function spaces. In IgA, those function spaces are either B-spline or NURBS
functions. In this paper, we focus on B-splines.
Let $p \in \mathbb{N}:={1,2,3,\dots}$ be a given spline degree. For ease of
notation, we assume that it is the same for all patches. 
For $n\in \mathbb{N}$, a $p$-open knot vector
\[
	\Xi=(\xi_1,\ldots,\xi_{n+p+1})
	=(\underbrace{\zeta_1,\ldots,\zeta_1}_{\displaystyle m_1 },
	\underbrace{\zeta_2,\ldots,\zeta_2}_{\displaystyle m_2},
	\ldots,
	\underbrace{\zeta_{N_Z},\ldots,\zeta_{N_Z}}_{\displaystyle m_{N_Z}})
\]
with multiplicities $m_1=m_{N_Z}=p+1$,
and $m_i\in\{1,\ldots,p\}$ for $i=2,\ldots,N_Z-1$ and breakpoints
$\zeta_1<\zeta_2<\cdots<\zeta_{N_Z}$ is the building block
for the B-spline basis $(B[p,\Xi,i])_{i=1}^n$. The individual
basis functions $B[p,\Xi,i]$ are defined via the
Cox-de~Boor formula, cf.~\cite[Eq.~(2.1) and (2.2)]{Cottrell:Hughes:Bazilevs}.
This basis spans the univariate spline space
\[
	S[p,\Xi]:=\text{span}\{B[p,\Xi,1],\ldots, B[p,\Xi,n] \}.
\]
Let $\Xi^{(k,1)}$ and $\Xi^{(k,2)}$ be two $p$-open
knot vectors over $(0,1)$. To get a multivariate spline space $\widehat{V}^{(k)}$ over $\widehat \Omega$, we tensorize the two univariate spline spaces. The transformation of $\widehat{V}^{(k)}$ to the physical domain is defined by the pull-back principle. We denote the resulting space by $V^{(k)}$. So, we define
\begin{equation}\label{eq:vkdef}
		\widehat{V}^{(k)}
					:= \{ v\in S[p,\Xi^{(k,1)}] \otimes S[p,\Xi^{(k,2)}]
							\;:\;
							v|_{\widehat \Gamma_D^{(k)}} = 0 \}	
		\quad\mbox{and}\quad
		V^{(k)} := \widehat{V}^{(k)} \circ G_k^{-1},
\end{equation}
where $v|_T$ denotes the restriction of $v$ to $T$ (trace operator).
The basis for the space $\widehat V^{(k)}$ consists of only those tensor-product basis functions over $\Omega^{(k)}$ that vanish on the pre-image of the Dirichlet boundary~$\widehat \Gamma_D^{(k)}$. Say, the total number of basis functions of $\widehat{V}^{(k)}$ is $N^{(k)}= N_\mathrm{I}^{(k)} + N_\Gamma^{(k)}$. The number $N_\mathrm{I}^{(k)}$ denotes the number of basis function that are supported only in the interior of the patch whereas $N_\Gamma^{(k)}$ accounts for the number of basis functions that contribute to the boundary of the patch. These definitions give rise to an ordered basis 
\begin{equation}\label{eq:basis:def}
	\begin{aligned}
		&\widehat \Phi^{(k)} := ( \widehat\phi_i^{(k)} )_{i=1}^{N^{(k)}},\\
		&\{ \widehat \phi_i^{(k)} \}
		= \{ \widehat \phi\,:\, \exists j_1,j_2\;:\;\widehat\phi(x,y)=B[p,\Xi^{(k,1)},j_1](x)\, B[p,\Xi^{(k,2)},j_2](y)\wedge
				\widehat \phi|_{\widehat \Gamma_D^{(k)}}  = 0 \},
		\\
		&	
		\widehat\phi^{(k)}_i|_{\partial \widehat \Omega}  = 0 \Leftrightarrow i \in\{1,\ldots, N_\mathrm{I}^{(k)}\},
		\quad
		\mbox{and}
		\quad
		\widehat\phi^{(k)}_i|_{\partial \widehat \Omega} \not = 0 \Leftrightarrow i \in N_\mathrm{I}^{(k)}+\{1,\ldots, N_\Gamma^{(k)}\}.
	\end{aligned}
\end{equation}
The pull-back principle gives a basis for $V^{(k)}$:
\[
		\Phi^{(k)} := ( \phi_i^{(k)} )_{i=1}^{N^{(k)}}
		\quad\mbox{and}\quad
		 \phi_{i}^{(k)} := \widehat \phi_{i}^{(k)} \circ G_k^{-1}.
\]
We assume in the following that the grids on each of the patches are quasi-uniform. 
\begin{assumption}
	\label{ass:quasiuniform}
	There are grid sizes $\widehat{h}_{k}>0$ for $k=1,\ldots,K$
	and a constant $C_3>0$ such that
	\[
			C_3 \, \widehat{h}_{k} \le
			\zeta_{i+1}^{(k,\delta)} - \zeta_i^{(k,\delta)} \le \widehat{h}_{k}
	\]
	holds for all $i=1,\ldots,N_Z^{(k,\delta)}-1$ and all $\delta=1,2$.
\end{assumption}

The corresponding grid size on the physical domain is defined via
$h_{k}:=\widehat{h}_{k} H_{k}$. For any two patches
$\Omega^{(k)}$ and $\Omega^{(\ell)}$ sharing an edge, we define
\[
	\widehat h_{k\ell} := \text{min} \{\widehat h_{k}, \widehat h_{\ell} \}
	\quad\mbox{and}\quad
	h_{k\ell} := \text{min} \{h_{k}, h_{\ell} \}.
\]

The product space of the local spaces $V^{(k)}$ gives the global approximation space
\begin{equation}\label{eq:vdef}
	V := \prod_{k=1}^{K} V^{(k)} := V^{(1)} \times \dots \times V^{(K)}.
\end{equation}
Using these function spaces, the dG discretization of the model problem~\eqref{continousProb} is given by: find $u
 = (u^{(1)},\cdots, u^{(K)})
 \in V$ such that 
\begin{align}
	\label{discreteVarProb}
	a_h(u,v)=\langle f,v\rangle
	\quad \mbox{for all}\quad v \in V,
\end{align}
where 
\begin{align*}
	a_h(u,v) &:= \sum_{k=1}^{K} \left( a^{(k)}(u,v) + m^{(k)}(u,v) + r^{(k)}(u,v) \right), \\
	a^{(k)}(u,v) &:= \int_{\Omega^{(k)}} \nabla u^{(k)} \cdot \nabla v^{(k)} \; \textrm{d}x, \\
	m^{(k)}(u,v) &:= \sum_{\ell \in \mathcal{N}_\Gamma(k)} \int_{\Gamma^{(k,\ell)}} \frac{1}{2} 
	\left( 
		\frac{\partial u^{(k)}}{\partial n_k}(v^{(\ell)} - v^{(k)}) + 
		\frac{\partial v^{(k)}}{\partial n_k}(u^{(\ell)} - u^{(k)})
	\right) \; \textrm{d}s, \\
	r^{(k)}(u,v) &:= \sum_{\ell \in \mathcal{N}_\Gamma(k)} \int_{\Gamma^{(k,\ell)}} \frac{\delta p^2}{h_{k\ell}}   
	(u^{(\ell)} - u^{(k)})(v^{(\ell)} - v^{(k)}) \; \textrm{d}s,\\
	\langle {f}, v \rangle
	&:=  
	\sum_{k=1}^{K} \int_{\Omega^{(k)}} f v^{(k)} \; \textrm{d}x,
\end{align*}
and $\delta > 0$ is some suitably chosen penalty parameter and $n_k$ is the outward unit normal vector of the patch $\Omega^{(k)}$.

Due to \cite[Theorem 8]{Takacs:2019b}, the
parameter $\delta$ can always be chosen independently of the spline degree $p$ and mesh sizes $h_k$ such that the bilinear form $a_h(\cdot,\cdot)$ is bounded and coercive in the dG-norm
\[
	\| v \|^2_{d} := d(v, v), \quad \text{ where } \quad
 	d(u,v) := \sum_{k=1}^{K} \left( a^{(k)}(u,v) + r^{(k)}(u,v) \right).
\]
Note that $\delta$ depends on the constant $C_1$ from Assumption~\ref{ass:nabla}.
Similar results have been shown previously, but up to the knowledge of the authors,
the dependence on $p$ has only been addressed in \cite[Theorem 8]{Takacs:2019b}.

Since we have coercivity and boundedness, the Lax-Milgram lemma ensures the
existence of a solution of~\eqref{discreteVarProb} and its uniqueness.
The solution to \eqref{discreteVarProb} is an approximation to the solution of the original problem \eqref{continousProb}, cf. \cite[Theorems 12 and 13]{Takacs:2019b}.

If we choose a basis for $V$, the discrete variational problem~\eqref{discreteVarProb}
can be rewritten in matrix-vector form. For the setup of a IETI method,
we not need this step since we directly work with the discrete variational
problem~\eqref{discreteVarProb}.

\section{The dG IETI-DP solver}
\label{sec:3}

In this section, we propose a IETI-DP solver for the variational
problem~\eqref{discreteVarProb}. As for any tearing and interconnecting
method, we have to introduce local spaces, cf., e.g.,~\cite{FarhatRoux:1991a}.
As it has been done previously for IETI methods, the local spaces are
chosen on a per-patch basis, cf.~\cite{KleissPechsteinJuttlerTomar:2012}
and later publications on IETI. In the case of dG, the choice is not completely
straight-forward. We follow the approach that has been proposed
in~\cite{DryjaGalvis:2013} and that has already been used in an IgA
context in~\cite{HoferLanger:2016a,Hofer:2016a} and others.

The patch-local subspace consists of the local functions on the particular
patch and of those functions from the neighboring patches required to realize
the bilinear forms $m^{(k)}(\cdot,\cdot)$ and $r^{(k)}(\cdot,\cdot)$.
We define the enriched function space
\begin{align*}
	V_{e}^{(k)}:= V^{(k)} \times \prod_{\ell \in \mathcal{N}_\Gamma(k)} V^{(k, \ell)},
\end{align*}
where $V^{(k,\ell)} := \{ v^{(\ell)}|_{\Gamma^{(k,\ell)}} \;:\; v^{(\ell)}\in V^{(\ell)}\} $ is the trace of $V^{(\ell)}$.
Correspondingly, we write
\begin{align}
	\label{def:representation}
	v_e^{(k)} = 
	\left(
		v^{(k)}, (v^{(k,\ell)})_{\ell\in \mathcal{N}_\Gamma(k)}
	\right)
\end{align}
where $v_e^{(k)} \in V_e^{(k)}$, $v^{(k)} \in V^{(k)}$ and $v^{(k,\ell)} \in V^{(k,\ell)}$. Note that the traces of the basis functions
\[
\{\phi^{(k,\ell)}_1,\ldots,\phi^{(k,\ell)}_{N^{(k,\ell)}}\} := 
\left\{ \phi_i^{(\ell)}|_{\Gamma^{(k,\ell)}} \;:\;
\phi_i^{(\ell)} \in \Phi^{(\ell)}, \; \phi_i^{(\ell)}|_{\Gamma^{(k,\ell)}} \not=0\right\}
\]
form a basis of $V^{(k,\ell)}$, which is uniquely defined up the the ordering of
the basis functions. We denote the basis by $\Phi^{(k,\ell)}:=
(\phi^{(k,\ell)}_i)_{i=1}^{N^{(k,\ell)}}$.

This definition can be interpreted as using artificial interfaces.
A visualization is given in Fig.~\ref{fig:ai}. The basis functions for each of
the patches are represented via different symbols. The location of the symbol
is the point where the corresponding basis function takes its maximum.
The symbols that are located on the edges and corners represent the only functions
that are supported on that edge or corner, respectively. Besides the patches
themselves, we have artificial interfaces. While they are co-located with
the (standard) interfaces, they are treated as separate entities on which the
corresponding trace spaces $V^{(k,\ell)}$ live. The basis functions living on
these trace spaces are marked with the same kind of symbol as the basis
function from the original space.

\begin{figure}
\begin{center}
   \begin{tikzpicture}
       \def\shift{0.5}
       \fill[gray!20] (-0.2,0) -- (1.5,0) -- (1.5,1.7) -- (-0.2,1.7);
       \fill[gray!20] (-0.2,-1.5-\shift) -- (1.5,-1.5-\shift) -- (1.5,-3.2-\shift) -- (-0.2,-3.2-\shift);
       \fill[gray!20] (4.7+\shift,0) -- (3.0+\shift,0) -- (3.0+\shift,1.7) -- (4.7+\shift,1.7);
       \fill[gray!20] (4.7+\shift,-1.5-\shift) -- (3.0+\shift,-1.5-\shift) -- (3.0+\shift,-3.2-\shift) -- (4.7+\shift,-3.2-\shift);

       \draw (-0.2,0) -- (1.5,0) -- (1.5,1.7) node at (.25,2) {$\Omega^{(1)}$};
       \draw (-0.2,-1.5-\shift) -- (1.5,-1.5-\shift) -- (1.5,-3.2-\shift) node at (.25,-3.5-\shift) {$\Omega^{(2)}$};
       \draw (4.7+\shift,0) -- (3.0+\shift,0) -- (3.0+\shift,1.7) node at (4.75+\shift,2) {$\Omega^{(3)}$};
       \draw (4.7+\shift,-1.5-\shift) -- (3.0+\shift,-1.5-\shift) -- (3.0+\shift,-3.2-\shift) node at (4.75+\shift,-3.5-\shift) {$\Omega^{(4)}$};

       \node at (-0.7,1.3) {$V^{(1)}$};
       \node at (-0.7, -0.4) {$V^{(1,2)}$};
       \node at (2.0, 2.0) {$V^{(1,3)}$};

       \node at (-0.7,-2.8-\shift) {$V^{(2)}$};
       \node at (-0.7,-1.1-\shift) {$V^{(2,1)}$};
       \node at (2.0,-3.5-\shift) {$V^{(2,4)}$};

       \node at (5.2+\shift,1.3) {$V^{(3)}$};
       \node at (2.6+\shift,2.0) {$V^{(3,1)}$};
       \node at (5.2+\shift,-0.4) {$V^{(3,4)}$};

       \node at (5.2+\shift,-2.8-\shift) {$V^{(4)}$};
       \node at (5.2+\shift,-1.1-\shift) {$V^{(4,3)}$};
       \node at (2.6+\shift,-3.5-\shift) {$V^{(4,2)}$};

       \draw (-0.2,-0.4) -- (1.5,-0.4) {};
       \draw (1.9,0) -- (1.9,1.7) {};

       \draw (3.0+\shift,-0.4) -- (4.7+\shift,-0.4) {};
       \draw (2.6+\shift,0) -- (2.6+\shift,1.7) {};

       \draw (-0.2,-1.1-\shift) -- (1.5,-1.1-\shift) {};
       \draw (1.9,-1.5-\shift) -- (1.9,-3.2-\shift) {};

       \draw (3.0+\shift,-1.1-\shift) -- (4.7+\shift,-1.1-\shift) {};
       \draw (2.6+\shift,-1.5-\shift) -- (2.6+\shift,-3.2-\shift) {};

       \draw (1.5,0) node[circle, fill, inner sep = 2pt] (A7) {};
       \draw (1.5,1.0) node[circle, fill, inner sep = 2pt] (A9) {};
       \draw (0.35,0.0) node[circle, fill, inner sep = 2pt] (A11) {};
       \draw (0.35,1.0) node[circle, fill, inner sep = 2pt] (A11) {};

       \draw (1.9,0) node[rectangle, fill, inner xsep = 2.5pt, inner ysep = 2.5pt] (A1) {};
       \draw (1.9,0.5) node[rectangle, fill, inner xsep = 2.5pt, inner ysep = 2.5pt] (A2) {};
       \draw (1.9,1.5) node[rectangle, fill, inner xsep = 2.5pt, inner ysep = 2.5pt] (A3) {};
       \draw (0.75,-0.4) node[diamond, fill, inner sep = 2pt] (A5) {};
       \draw (1.5,-0.4) node[diamond, fill, inner sep = 2pt] (A6) {};

       \draw (3.0+\shift,0) node[rectangle, fill, inner xsep = 2.5pt, inner ysep = 2.5pt] (B7) {};
       \draw (3.0+\shift,0.5) node[rectangle, fill, inner xsep = 2.5pt, inner ysep = 2.5pt] (B8) {};
       \draw (3.0+\shift,1.5) node[rectangle, fill, inner xsep = 2.5pt, inner ysep = 2.5pt] (B9) {};
       \draw (3.75+\shift,0) node[rectangle, fill, inner xsep = 2.5pt, inner ysep = 2.5pt] (B10) {};
       \draw (3.75+\shift,0.5) node[rectangle, fill, inner xsep = 2.5pt, inner ysep = 2.5pt] (B10) {};
       \draw (3.75+\shift,1.5) node[rectangle, fill, inner xsep = 2.5pt, inner ysep = 2.5pt] (B10) {};
       \draw (4.5+\shift,0) node[rectangle, fill, inner xsep = 2.5pt, inner ysep = 2.5pt] (B11) {};
       \draw (4.5+\shift,0.5) node[rectangle, fill, inner xsep = 2.5pt, inner ysep = 2.5pt] (B11) {};
       \draw (4.5+\shift,1.5) node[rectangle, fill, inner xsep = 2.5pt, inner ysep = 2.5pt] (B11) {};

       \draw (2.6+\shift,0) node[circle, fill, inner sep = 2pt] (B1) {};
       \draw (2.6+\shift,1.0) node[circle, fill, inner sep = 2pt] (B3) {};
       \draw (3.0+\shift,-0.4) node[star, fill, inner sep = 2pt] (B4) {};
       \draw (3.75+\shift,-0.4) node[star, fill, inner sep = 2pt] (B5) {};
       \draw (4.5+\shift,-0.4) node[star, fill, inner sep = 2pt] (B6) {};

       \draw (1.5,-1.5-\shift) node[diamond, fill, inner sep = 2pt] (C7) {};
       \draw (1.5,-2.2-\shift) node[diamond, fill, inner sep = 2pt] (C8) {};
       \draw (1.5,-3.0-\shift) node[diamond, fill, inner sep = 2pt] (C9) {};
       \draw (0.75,-1.5-\shift) node[diamond, fill, inner sep = 2pt] (C11) {};
       \draw (0.75,-2.2-\shift) node[diamond, fill, inner sep = 2pt] (C11) {};
       \draw (0.75,-3.0-\shift) node[diamond, fill, inner sep = 2pt] (C11) {};

       \draw (1.9,-1.5-\shift) node[star, fill, inner sep = 2pt] (C1) {};
       \draw (1.9,-2.0-\shift) node[star, fill, inner sep = 2pt] (C2) {};
       \draw (1.9,-2.75-\shift) node[star, fill, inner sep = 2pt] (C3) {};
       \draw (0.35,-1.1-\shift) node[circle, fill, inner sep = 2pt] (C5) {};
       \draw (1.5,-1.1-\shift) node[circle, fill, inner sep = 2pt] (C6) {};

       \draw (3.0+\shift,-1.5-\shift) node[star, fill, inner sep = 2pt] (D7) {};
       \draw (3.0+\shift,-2.0-\shift) node[star, fill, inner sep = 2pt] (D8) {};
       \draw (3.0+\shift,-2.75-\shift) node[star, fill, inner sep = 2pt] (D9) {};
       \draw (3.75+\shift,-1.5-\shift) node[star, fill, inner sep = 2pt] (D10) {};
       \draw (3.75+\shift,-2.0-\shift) node[star, fill, inner sep = 2pt] (D10) {};
       \draw (3.75+\shift,-2.75-\shift) node[star, fill, inner sep = 2pt] (D10) {};
       \draw (4.5+\shift,-1.5-\shift) node[star, fill, inner sep = 2pt] (D11) {};
       \draw (4.5+\shift,-2.0-\shift) node[star, fill, inner sep = 2pt] (D11) {};
       \draw (4.5+\shift,-2.75-\shift) node[star, fill, inner sep = 2pt] (D11) {};

       \draw (2.6+\shift,-1.5-\shift) node[diamond, fill, inner sep = 2pt] (D1) {};
       \draw (2.6+\shift,-2.2-\shift) node[diamond, fill, inner sep = 2pt] (D2) {};
       \draw (2.6+\shift,-3.0-\shift) node[diamond, fill, inner sep = 2pt] (D3) {};
       \draw (3.0+\shift,-1.1-\shift) node[rectangle, fill, inner xsep = 2.5pt, inner ysep = 2.5pt] (D4) {};
       \draw (3.75+\shift,-1.1-\shift) node[rectangle, fill, inner xsep = 2.5pt, inner ysep = 2.5pt] (D5) {};
       \draw (4.5+\shift,-1.1-\shift) node[rectangle, fill, inner xsep = 2.5pt, inner ysep = 2.5pt] (D6) {};
   \end{tikzpicture}
   \captionof{figure}{Local spaces with artificial interfaces \label{fig:ai}}
\end{center}
\end{figure}
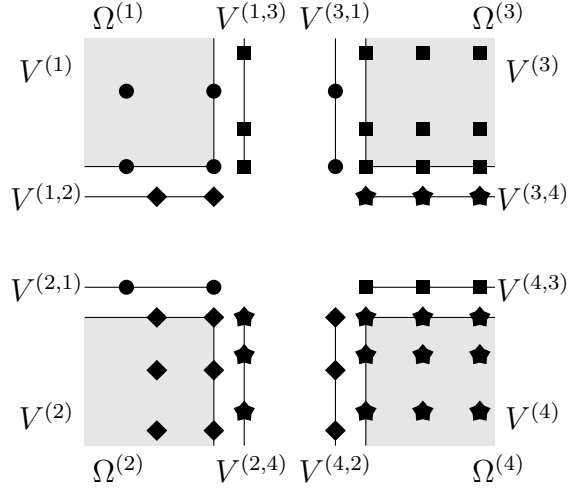

Following the structure of \eqref{def:representation}, we define a basis
$\Phi_e^{(k)} := ( \phi_{e,i}^{(k)} )_{i=1}^{N_e^{(k)}}$
with $N_e^{(k)}:=\dim V_e^{(k)}$
for the space $V_e^{(k)}$ as follows. We start with the basis functions from
the basis $\Phi^{(k)}$:
\[
		\phi_{e,i}^{(k)} := (\phi_{i}^{(k)},(0,\ldots,0))
		\quad \mbox{for}\quad i=1,\ldots,N^{(k)}.
\]
Then, there follow the basis functions from the neighboring patches. Let
$\mathcal{N}_\Gamma(k) = \{\ell_1, \ell_2, \ldots, \ell_L \}$. We define
\begin{equation}\label{eq:def extended basis functions}
\begin{aligned}
		\phi_{e,\Upsilon^{(k,\ell_1)}_i}^{(k)} &:= (0,(\phi_{i}^{(k,\ell_1)},0\ldots,0))
		\\&   \mbox{with}\; \Upsilon^{(k,\ell_1)}_i:=N^{(k)}+i
	&	\quad \mbox{for} \; i=1,\ldots,N^{(k,\ell_1)},\\
		\phi_{e,\Upsilon^{(k,\ell_2)}_i}^{(k)} &:= (0,(0,\phi_{i}^{(k,\ell_2)},0\ldots,0))
		\\&  \mbox{with}\; \Upsilon^{(k,\ell_2)}_i:=N^{(k)}+N^{(k,\ell_1)}+i
	&	\quad \mbox{for} \; i=1,\ldots,N^{(k,\ell_2)},\\
		&\vdots\\
		\phi_{e,\Upsilon^{(k,\ell_L)}_i}^{(k)} &:= (0,(0,\ldots,0,\phi_{i}^{(k,\ell_{L})}))
		\\&  \mbox{with}\; \Upsilon^{(k,\ell_L)}_i:=N^{(k)}+ 
		N^{(k,\ell_1)}+\cdots+N^{(k,\ell_{L-1})}+i
	&	\quad \mbox{for} \; i=1,\ldots,N^{(k,\ell_{L})}.
\end{aligned}
\end{equation}
From this construction and from~\eqref{eq:basis:def}, we know that the first
$N_I^{(k)}$ basis functions of the basis $\Phi_e^{(k)}$
live in the interior of the patch $\Omega^{(k)}$.
The following $N_{\Gamma}^{(k)}$ basis functions live on the interfaces. Finally,
the remaining basis functions live on the artificial interfaces.

On the spaces $V_e^{(k)}$, we define the bilinear forms $a_e^{(k)}(\cdot,\cdot)$ and
$d_e^{(k)}(\cdot,\cdot)$ and the linear functional $\langle f_e^{(k)},\cdot\rangle$
analogous to the global objects
$a(\cdot,\cdot)$, $d(\cdot,\cdot)$, and $\langle f,\cdot\rangle$ by
\[
\begin{aligned}
	a_e^{(k)}(u_e^{(k)},v_e^{(k)}) &:= a^{(k)}(u_e^{(k)},v_e^{(k)}) + m^{(k)}(u_e^{(k)},v_e^{(k)}) + r^{(k)}(u_e^{(k)},v_e^{(k)}), \\
	d_e^{(k)}(u_e^{(k)},v_e^{(k)}) &:= a^{(k)}(u_e^{(k)},v_e^{(k)}) + r^{(k)}(u_e^{(k)},v_e^{(k)}), \\
	\langle f_e^{(k)},v_e^{(k)}\rangle & := \int_{\Omega^{(k)}} f v^{(k)} \mathrm dx,
\end{aligned}
\]
where we write with a slight abuse of notation
\[
\begin{aligned}
	a^{(k)}(u_e^{(k)},v_e^{(k)}) &:= \int_{\Omega^{(k)}} \nabla u^{(k)} \cdot \nabla v^{(k)} \; \textrm{d}x, \\
	m^{(k)}(u_e^{(k)},v_e^{(k)}) &:= \sum_{\ell \in \mathcal{N}_\Gamma(k)} \int_{\Gamma^{(k,\ell)}} \frac{1}{2} 
	\left( 
	\frac{\partial u^{(k)}}{\partial n_k}(v^{(k,\ell)} - v^{(k)}) + 
	\frac{\partial v^{(k)}}{\partial n_k}(u^{(k,\ell)} - u^{(k)})
	\right) \; \textrm{d}s, \\
	r^{(k)}(u_e^{(k)},v_e^{(k)}) &:= \sum_{\ell \in \mathcal{N}_\Gamma(k)} \int_{\Gamma^{(k,\ell)}} \frac{\delta p^2}{h_{k\ell}}   
	(u^{(k,\ell)} - u^{(k)})(v^{(k,\ell)} - v^{(k)}) \; \textrm{d}s.
\end{aligned}
\]
The bilinear form $d_e^{(k)}(\cdot,\cdot)$ induces the local dG-norm
$\| \cdot \|^2_{d_e^{(k)}} := d_e^{(k)}(\cdot, \cdot)$, the bilinear form
$a_e^{(k)}(\cdot,\cdot)$ the local energy norm
$\| \cdot \|^2_{a_e^{(k)}} := a_e^{(k)}(\cdot, \cdot)$.

By discretizing $a_e^{(k)}(\cdot,\cdot)$ and $\langle f_e^{(k)},\cdot \rangle$
using the basis $\Phi_e^{(k)}$, we obtain the local system
\begin{equation}\label{linsys:local}
			A^{(k)}\, \underline{u}_e^{(k)} = \underline{f}_e^{(k)},
\end{equation}
where $A^{(k)} = [a_e^{(k)}(\phi_{e,j}^{(k)},\phi_{e,i}^{(k)})]_{i,j=1}^{N_e^{(k)}}$
and
$\underline{f}_e^{(k)} = [ \langle f_e^{(k)},\phi_{e,i}^{(k)}\rangle]_{i=1}^{N_e^{(k)}}$. The vector $\underline{u}_e^{(k)}= [u_{e,i}^{(k)}]_{i=1}^{N_e^{(k)}}$
is the coefficient vector representing the function $u_e^{(k)} =\sum_{i=1}^{N_e^{(k)}}
u_{e,i}^{(k)} \phi_{e,i}^{(k)}$.

We subdivide the stiffness matrix and the load vector into blocks
\[
		A^{(k)}
		=
		\left(
		\begin{array}{cccc}
			A_{\mathrm I \mathrm I}^{(k)} & A_{\mathrm I\Gamma}^{(k)}\\
			A_{\Gamma \mathrm I}^{(k)} & A_{\Gamma\Gamma}^{(k)}\\
		\end{array}
		\right)
		\quad
		\mbox{and}
		\quad
		\underline{f}_e^{(k)}
		=
		\left(
		\begin{array}{c}
			\underline{f}_{\mathrm I}^{(k)}\\\underline{f}_{\Gamma}^{(k)}
		\end{array}
		\right),
\]
where the first row and the first column correspond to the first $N_I^{(k)}$ basis
functions, i.e., to those supported in the interior of the patch $\Omega^{(k)}$.
So, the remainder accounts for both the standard interfaces and the artificial
interfaces.

Next, we build the Schur complement system of \eqref{linsys:local} with respect to the interface degrees of freedom:
\begin{equation}\label{eq:s:sys}
			S^{(k)} \underline{w}_e^{(k)} = \underline{g}_e^{(k)},
\end{equation}
where
\begin{equation}\label{eq:s:def}
	S^{(k)} := A_{\Gamma \Gamma}^{(k)} - A_{\Gamma \mathrm I}^{(k)}(A_{\mathrm I\mathrm I}^{(k)})^{-1} A_{\mathrm I\Gamma}^{(k)}
	 \quad\mbox{and}\quad
	\underline{g}_e^{(k)} := \underline{f}_{\Gamma}^{(k)}
- A_{\Gamma \mathrm I}^{(k)} (A_{\mathrm I\mathrm I}^{(k)})^{-1}  \underline{f}_{\mathrm I}^{(k)}.
\end{equation}

Once $\underline{w}_e^{(k)}$ has been computed, we get the solution $\underline{u}_e^{(k)}$ of the system \eqref{linsys:local} by 
\begin{equation}\label{eq:inner:def}
	\underline{u}_e^{(k)}
		=
		\left(
		\begin{array}{c}
			\underline{u}_{\mathrm I}^{(k)}\\\underline{u}_{\Gamma}^{(k)}
		\end{array}
		\right)
		=
		\left(
		\begin{array}{c}
			(A_{\mathrm I\mathrm I}^{(k)})^{-1}
			A_{\mathrm I \Gamma}^{(k)}
			\underline{w}_e^{(k)}\\\underline{w}_e^{(k)}
		\end{array}
		\right).
\end{equation}
The collection of all linear systems~\eqref{eq:s:sys} for all patches yields a block diagonal linear system
\begin{equation}\label{eq:schursys}
	S \underline w = \underline g,
\end{equation}
where
\[
	S := \left(
		\begin{array}{ccc}
			S^{(1)} \\ & \ddots \\ && S^{(K)}
		\end{array}	
	\right),
	\quad 
	\underline w := \left(
		\begin{array}{c}
			\underline w_e^{(1)} \\  \vdots \\ \underline w_e^{(K)}
		\end{array}	
	\right)
	\quad \mbox{and}\quad
	\underline g := \left(
		\begin{array}{c}
			\underline g_e^{(1)} \\  \vdots \\ \underline g_e^{(K)}
		\end{array}	
	\right).
\]

In the following, we introduce the building blocks in order to rewrite the Schur
complement system in a variational setting.
First, we define function spaces on the skeleton:
\[
	W := \prod_{k=1}^K W_e^{(k)},
	\quad\mbox{where}\quad
	W_e^{(k)} := W^{(k)}
	\times \prod_{\ell \in \mathcal{N}_\Gamma(k)} W^{(k,\ell)},
\]
$W^{(k)}:=\{ v|_{\partial \Omega^{(k)}} \,:\, v \in V^{(k)}\}$
and $W^{(k,\ell)}:=V^{(k,\ell)}$. Analogously
to \eqref{def:representation}, an extended function $w_e^{(k)}\in W_e^{(k)}$ has the
form 
\[
	w_e^{(k)} = \left(w^{(k)}, \left(w^{(k,\ell)}\right)_{\ell\in \mathcal{N}_\Gamma(k)} \right),
\] 
where $w^{(k)} \in W^{(k)}$ and $w^{(k,\ell)}\in W^{(k,\ell)}$. A basis for
$W_e^{(k)}$ is given by
\[
\breve{\Phi}_e^{(k)}:=
\left(\phi_{e,N_I^{(k)}+i}^{(k)}\right)_{i=1}^{\breve{N}_e^{(k)}}
\quad\mbox{with}\quad
\breve{N}_e^{(k)}:=N_e^{(k)} - N_I^{(k)}.
\]
As above, the coefficient vector $\underline{w}_e^{(k)}=[w_{e, i}^{(k)}]_{i=1}^{\breve{N}_e^{(k)}}$ represents the function
\begin{equation}\label{eq:webasis}
		w_e^{(k)} = \sum_{i=1}^{\breve{N}_{e}^{(k)}} w_{e, i}^{(k)} \phi_{e,N_I^{(k)}+i}^{(k)}
\end{equation} with respect to that basis.

The Schur complements realize the discrete harmonic extension
$\mathcal{H}_A^{(k)}:W_e^{(k)}\rightarrow V^{(k)}_e$ with respect to
the energy norm $\|\cdot\|_{a_e^{(k)}}$,
which is defined as follows.
For any $w_e^{(k)} = (w^{(k)},(w^{(k,\ell)})_{\ell\in \mathcal{N}_\Gamma(k)})$,
we have
$
		\mathcal{H}_A^{(k)} w_e^{(k)} =
			\left(u^{(k)},(w^{(k,\ell)})_{\ell\in \mathcal{N}_\Gamma(k)}\right) 
$,
where $u^{(k)}$ is such that
\begin{equation}
	\label{discrete harmonic extension}
	\begin{aligned}
		u^{(k)}&=w^{(k)} \quad \mbox{on} \quad \partial \Omega^{(k)},\\
		a_e^{(k)}\left(
		\left(u^{(k)},(w^{(k,\ell)})_{\ell\in \mathcal{N}_\Gamma(k)}\right),
		\left(v^{(k)},0^{|\mathcal{N}_\Gamma(k)|}\right)
		\right)&=0 
		\quad \mbox{for all} \quad
		v^{(k)}\in V_{0}^{(k)},
	\end{aligned}
\end{equation}
where $V_{0}^{(k)} :=V^{(k)} \cap H^1_0(\Omega^{(k)})$.
Note that the standard discrete harmonic extension
$\mathcal{H}_h^{(k)}: W^{(k)} \rightarrow V^{(k)}$ is defined by
$\mathcal{H}_h^{(k)} w^{(k)}= u^{(k)}$, where $u^{(k)}$ is such that
\begin{equation}
  \label{standard discrete harmonic extension}
	\begin{aligned}
	u^{(k)} &= w^{(k)} \quad \mbox{on} \quad \partial \Omega^{(k)},\\
	a^{(k)}(u^{(k)},v^{(k)})&=0
	\quad\mbox{for all}\quad
	v^{(k)}\in V_{0}^{(k)}.
	\end{aligned}
\end{equation}

In variational form, problem~\eqref{eq:schursys} reads as follows: find
$w =(w_e^{(1)},\ldots,w_e^{(K)}) \in W$ such that 
\[
	\underbrace{\sum_{k=1}^K a_e^{(k)} ( \mathcal{H}_A^{(k)} w_e^{(k)}, \mathcal{H}_A^{(k)} q_e^{(k)} )}
	_{\displaystyle s(w,q):=}
	=
	\underbrace{
	\sum_{k=1}^K \langle f, \mathcal{H}_A^{(k)} q_e^{(k)} \rangle}
	_{\displaystyle \langle g,q\rangle:=}
	\quad\mbox{for all}\quad
	q = (q_e^{(1)},\ldots,q_e^{(K)}) \in W.
\]

The next step is the introduction of constraints that yield the continuity
of the solution. For any two patches $\Omega^{(k)}$ and $\Omega^{(\ell)}$
sharing an edge $\Gamma^{(k,\ell)}$ and any basis function 
$\phi_i^{(k)}$ supported on that edge, we introduce a constraint
\begin{equation}
\label{eq:def:b1}
			w^{(k)}_{e,i^*} - w^{(\ell)}_{e,j^*} = 0
			\quad\mbox{with}\quad
			i^*=i-N_I^{(k)}
			\quad\mbox{and}\quad
			j^*=\Upsilon_i^{(\ell,k)}-N_I^{(\ell)},
\end{equation}
where 
$\Upsilon_i^{(\ell,k)}$ is as in~\eqref{eq:def extended basis functions}
and $w^{(k)}_{e,i^*}$ and $w^{(k)}_{e,j^*}$
are the coefficients of the functions $w_e^{(k)}$ and $w_e^{(\ell)}$.
The index shifts by $N_I^{(k)}$ and $N_I^{(\ell)}$ account for the restriction to $W_e^{(k)}$ and $W_e^{(\ell)}$, cf.~\eqref{eq:webasis}.
This constraint ensures that the function values of the solution 
on the artificial interfaces coincide with the function values of the solution
on the corresponding standard interfaces.

If the corner values are chosen as primal degrees of freedom (Alg.~A and C),
we omit the constraints for the basis functions $\phi_i^{(k)}$ that are supported on
a corner of $\Omega^{(k)}$. A visualization of the constraints is given
in Figure~\ref{fig:ommiting}.

If only the edge averages are chosen as primal degrees of freedom (Alg.~B), we realize the constraints at the corners in a fully redundant way. This means
that besides the constraints~\eqref{eq:def:b1}, for every basis function
$\phi_i^{(k)}$ that is supported on a corner of $\Omega^{(k)}$, 
we additionally introduce constraints of the form
\begin{equation}
\label{eq:def:b2}
			w^{(\ell_1)}_{e,j_1^*} - w^{(\ell_2)}_{e,j_2^*} = 0
			\quad\mbox{with}\quad
			j_1^*=\Upsilon_i^{(\ell_1,k)}-N_I^{(\ell_1)}
			\quad\mbox{and}\quad
			j_2^*=\Upsilon_i^{(\ell_2,k)}-N_I^{(\ell_2)}
\end{equation}
where $\ell_1 < \ell_2$ such that the said corner lies
between the edges $\Gamma^{(k,\ell_1)}$ and $\Gamma^{(k,\ell_2)}$.
These additional constraints also ensure the continuity between the different
artificial interfaces.
A visualization of this case is given in Figure~\ref{fig:fully}. 

We collect the constraints of the form~\eqref{eq:def:b1} and \eqref{eq:def:b2} in
a matrix $B$ such that the constraints are equivalent to
\[
	B \, \underline w=0.
\]
This is done such that each row of the matrix corresponds to one constraint, i.e.,
each row has only two non-zero entries, one with value $1$ and one with value $-1$. 
The matrix $B$ can also be represented with patch-local contributions $B^{(1)}, \dots, B^{(K)}$ as $B = (B^{(1)} \cdots B^{(K)})$.
\begin{figure}
\begin{center}
\begin{minipage}{0.35\linewidth}
	\begin{tikzpicture}
	\fill[gray!20] (-0.2,0) -- (1.5,0) -- (1.5,1.7) -- (-0.2,1.7);
	\fill[gray!20] (-0.2,-1.5) -- (1.5,-1.5) -- (1.5,-3.2) -- (-0.2,-3.2);
	\fill[gray!20] (4.7,0) -- (3.0,0) -- (3.0,1.7) -- (4.7,1.7);
	\fill[gray!20] (4.7,-1.5) -- (3.0,-1.5) -- (3.0,-3.2) -- (4.7,-3.2);
	
	\draw (-0.2,0) -- (1.5,0) -- (1.5,1.7) node at (0.7,0.8) {$\Omega^{(1)}$}; 
	\draw (-0.2,-1.5) -- (1.5,-1.5) -- (1.5,-3.2) node at (0.7,-2.35) {$\Omega^{(2)}$}; 
	\draw (4.7,0) -- (3.0,0) -- (3.0,1.7) node at (3.9,0.8) {$\Omega^{(3)}$}; 
	\draw (4.7,-1.5) -- (3.0,-1.5) -- (3.0,-3.2) node at (3.9,-2.35) {$\Omega^{(4)}$}; 
	
	\draw (-0.2,-0.4) -- (1.5,-0.4) {};
	\draw (1.9,0) -- (1.9,1.7) {};
	
	\draw (3.0,-0.4) -- (4.7,-0.4) {};
	\draw (2.6,0) -- (2.6,1.7) {};
	
	\draw (-0.2,-1.1) -- (1.5,-1.1) {};
	\draw (1.9,-1.5) -- (1.9,-3.2) {};
	
	\draw (3.0,-1.1) -- (4.7,-1.1) {};
	\draw (2.6,-1.5) -- (2.6,-3.2) {};
	
	\draw (1.5,0) node[circle, fill, inner sep = 2pt] (A7) {};
	\draw (1.5,1.0) node[circle, fill, inner sep = 2pt] (A9) {};
	\draw (0.35,0.0) node[circle, fill, inner sep = 2pt] (A11) {};
	
	\draw (1.9,0) node[rectangle, fill, inner xsep = 2.5pt, inner ysep = 2.5pt] (A1) {};
	\draw (1.9,0.5) node[rectangle, fill, inner xsep = 2.5pt, inner ysep = 2.5pt] (A2) {};
	\draw (1.9,1.5) node[rectangle, fill, inner xsep = 2.5pt, inner ysep = 2.5pt] (A3) {};
	\draw (0.75,-0.4) node[diamond, fill, inner sep = 2pt] (A5) {};
	\draw (1.5,-0.4) node[diamond, fill, inner sep = 2pt] (A6) {};

	\draw (3.0,0) node[rectangle, fill, inner xsep = 2.5pt, inner ysep = 2.5pt] (B7) {};
	\draw (3.0,0.5) node[rectangle, fill, inner xsep = 2.5pt, inner ysep = 2.5pt] (B8) {};
	\draw (3.0,1.5) node[rectangle, fill, inner xsep = 2.5pt, inner ysep = 2.5pt] (B9) {};
	\draw (3.75,0) node[rectangle, fill, inner xsep = 2.5pt, inner ysep = 2.5pt] (B10) {};
	\draw (4.5,0) node[rectangle, fill, inner xsep = 2.5pt, inner ysep = 2.5pt] (B11) {};

	\draw (2.6,0) node[circle, fill, inner sep = 2pt] (B1) {};
	\draw (2.6,1.0) node[circle, fill, inner sep = 2pt] (B3) {};
	\draw (3.0,-0.4) node[star, fill, inner sep = 2pt] (B4) {};
	\draw (3.75,-0.4) node[star, fill, inner sep = 2pt] (B5) {};
	\draw (4.5,-0.4) node[star, fill, inner sep = 2pt] (B6) {};
	
	\draw (1.5,-1.5) node[diamond, fill, inner sep = 2pt] (C7) {};
	\draw (1.5,-2.2) node[diamond, fill, inner sep = 2pt] (C8) {};
	\draw (1.5,-3.0) node[diamond, fill, inner sep = 2pt] (C9) {};
	\draw (0.75,-1.5) node[diamond, fill, inner sep = 2pt] (C11) {};
	
	\draw (1.9,-1.5) node[star, fill, inner sep = 2pt] (C1) {};
	\draw (1.9,-2.0) node[star, fill, inner sep = 2pt] (C2) {};
	\draw (1.9,-2.75) node[star, fill, inner sep = 2pt] (C3) {};
	\draw (0.35,-1.1) node[circle, fill, inner sep = 2pt] (C5) {};
	\draw (1.5,-1.1) node[circle, fill, inner sep = 2pt] (C6) {};
	
	\draw (3.0,-1.5) node[star, fill, inner sep = 2pt] (D7) {};
	\draw (3.0,-2.0) node[star, fill, inner sep = 2pt] (D8) {};
	\draw (3.0,-2.75) node[star, fill, inner sep = 2pt] (D9) {};
	\draw (3.75,-1.5) node[star, fill, inner sep = 2pt] (D10) {};
	\draw (4.5,-1.5) node[star, fill, inner sep = 2pt] (D11) {};
	
	\draw (2.6,-1.5) node[diamond, fill, inner sep = 2pt] (D1) {};
	\draw (2.6,-2.2) node[diamond, fill, inner sep = 2pt] (D2) {};
	\draw (2.6,-3.0) node[diamond, fill, inner sep = 2pt] (D3) {};
	\draw (3.0,-1.1) node[shape=rectangle, fill, inner xsep = 2.5pt, inner ysep = 2.5pt] (D4) {};
	\draw (3.75,-1.1) node[rectangle, fill, inner xsep = 2.5pt, inner ysep = 2.5pt] (D5) {};
	\draw (4.5,-1.1) node[rectangle, fill, inner xsep = 2.5pt, inner ysep = 2.5pt] (D6) {};
	
	\draw[<->, line width = 1pt, latex-latex, bend left]
	(A2) edge (B8) (A3) edge (B9)
	(A5) edge (C11)
	(B6) edge (D11) (B5) edge (D10)
	(C2) edge (D8) (C3) edge (D9);
	
	\draw[<->, line width = 1pt, latex-latex, bend right]
	(A9) edge (B3)
	(A11) edge (C5)
	(B11) edge (D6) (B10) edge (D5)
	(C8) edge (D2) (C9) edge (D3);
	\end{tikzpicture}
	\captionof{figure}{Omitting vertices\\ (Alg.~A and C) \label{fig:ommiting}}
\end{minipage}\hspace{4em}
\begin{minipage}{0.35\linewidth}
	\begin{tikzpicture}
	\fill[gray!20] (-0.2,0) -- (1.5,0) -- (1.5,1.7) -- (-0.2,1.7);
	\fill[gray!20] (-0.2,-1.5) -- (1.5,-1.5) -- (1.5,-3.2) -- (-0.2,-3.2);
	\fill[gray!20] (4.7,0) -- (3.0,0) -- (3.0,1.7) -- (4.7,1.7);
	\fill[gray!20] (4.7,-1.5) -- (3.0,-1.5) -- (3.0,-3.2) -- (4.7,-3.2);
	
	\draw (-0.2,0) -- (1.5,0) -- (1.5,1.7) node at (0.7,0.8) {$\Omega^{(1)}$}; 
	\draw (-0.2,-1.5) -- (1.5,-1.5) -- (1.5,-3.2) node at (0.7,-2.35) {$\Omega^{(2)}$}; 
	\draw (4.7,0) -- (3.0,0) -- (3.0,1.7) node at (3.9,0.8) {$\Omega^{(3)}$}; 
	\draw (4.7,-1.5) -- (3.0,-1.5) -- (3.0,-3.2) node at (3.9,-2.35) {$\Omega^{(4)}$}; 
	
	\draw (-0.2,-0.4) -- (1.5,-0.4) {};
	\draw (1.9,0) -- (1.9,1.7) {};
	
	\draw (3.0,-0.4) -- (4.7,-0.4) {};
	\draw (2.6,0) -- (2.6,1.7) {};
	
	\draw (-0.2,-1.1) -- (1.5,-1.1) {};
	\draw (1.9,-1.5) -- (1.9,-3.2) {};
	
	\draw (3.0,-1.1) -- (4.7,-1.1) {};
	\draw (2.6,-1.5) -- (2.6,-3.2) {};
	
	\draw (1.5,0) node[circle, fill, inner sep = 2pt] (A7) {};
	\draw (1.5,1.0) node[circle, fill, inner sep = 2pt] (A9) {};
	\draw (0.35,0.0) node[circle, fill, inner sep = 2pt] (A11) {};
	
	\draw (1.9,0) node[rectangle, fill, inner xsep = 2.5pt, inner ysep = 2.5pt] (A1) {};
	\draw (1.9,0.5) node[rectangle, fill, inner xsep = 2.5pt, inner ysep = 2.5pt] (A2) {};
	\draw (1.9,1.5) node[rectangle, fill, inner xsep = 2.5pt, inner ysep = 2.5pt] (A3) {};
	\draw (0.75,-0.4) node[diamond, fill, inner sep = 2pt] (A5) {};
	\draw (1.5,-0.4) node[diamond, fill, inner sep = 2pt] (A6) {};

	\draw (3.0,0) node[rectangle, fill, inner xsep = 2.5pt, inner ysep = 2.5pt] (B7) {};
	\draw (3.0,0.5) node[rectangle, fill, inner xsep = 2.5pt, inner ysep = 2.5pt] (B8) {};
	\draw (3.0,1.5) node[rectangle, fill, inner xsep = 2.5pt, inner ysep = 2.5pt] (B9) {};
	\draw (3.75,0) node[rectangle, fill, inner xsep = 2.5pt, inner ysep = 2.5pt] (B10) {};
	\draw (4.5,0) node[rectangle, fill, inner xsep = 2.5pt, inner ysep = 2.5pt] (B11) {};
	
	\draw (2.6,0) node[circle, fill, inner sep = 2pt] (B1) {};
	\draw (2.6,1.0) node[circle, fill, inner sep = 2pt] (B3) {};
	\draw (3.0,-0.4) node[star, fill, inner sep = 2pt] (B4) {};
	\draw (3.75,-0.4) node[star, fill, inner sep = 2pt] (B5) {};
	\draw (4.5,-0.4) node[star, fill, inner sep = 2pt] (B6) {};
	
	\draw (1.5,-1.5) node[diamond, fill, inner sep = 2pt] (C7) {};
	\draw (1.5,-2.2) node[diamond, fill, inner sep = 2pt] (C8) {};
	\draw (1.5,-3.0) node[diamond, fill, inner sep = 2pt] (C9) {};
	\draw (0.75,-1.5) node[diamond, fill, inner sep = 2pt] (C11) {};
	
	\draw (1.9,-1.5) node[star, fill, inner sep = 2pt] (C1) {};
	\draw (1.9,-2.0) node[star, fill, inner sep = 2pt] (C2) {};
	\draw (1.9,-2.75) node[star, fill, inner sep = 2pt] (C3) {};
	\draw (0.35,-1.1) node[circle, fill, inner sep = 2pt] (C5) {};
	\draw (1.5,-1.1) node[circle, fill, inner sep = 2pt] (C6) {};
	
	\draw (3.0,-1.5) node[star, fill, inner sep = 2pt] (D7) {};
	\draw (3.0,-2.0) node[star, fill, inner sep = 2pt] (D8) {};
	\draw (3.0,-2.75) node[star, fill, inner sep = 2pt] (D9) {};
	\draw (3.75,-1.5) node[star, fill, inner sep = 2pt] (D10) {};
	\draw (4.5,-1.5) node[star, fill, inner sep = 2pt] (D11) {};
	
	\draw (2.6,-1.5) node[diamond, fill, inner sep = 2pt] (D1) {};
	\draw (2.6,-2.2) node[diamond, fill, inner sep = 2pt] (D2) {};
	\draw (2.6,-3.0) node[diamond, fill, inner sep = 2pt] (D3) {};
	\draw (3.0,-1.1) node[rectangle, fill, inner xsep = 2.5pt, inner ysep = 2.5pt] (D4) {};
	\draw (3.75,-1.1) node[rectangle, fill, inner xsep = 2.5pt, inner ysep = 2.5pt] (D5) {};
	\draw (4.5,-1.1) node[rectangle, fill, inner xsep = 2.5pt, inner ysep = 2.5pt] (D6) {};
	
	\draw[<->, line width = 1pt, latex-latex]
	(A6) edge (D1) (A1) edge (D4) (B4) edge (C1) (B1) edge (C6);
	
	\draw[<->, line width = 1pt, latex-latex, bend left]
	(A1) edge (B7) 
	(A2) edge (B8) 
	(A3) edge (B9)
	(A5) edge (C11) (A6) edge (C7)
	(B6) edge (D11) (B5) edge (D10) (B4) edge (D7)
	(C2) edge (D8) (C3) edge (D9) (C1) edge (D7);
	
	\draw[<->, line width = 1pt, latex-latex, bend right]
	(A7) edge (B1) 
	(A9) edge (B3)
	(A11) edge (C5) (A7) edge (C6)
	(B11) edge (D6) (B10) edge (D5) (B7) edge (D4)
	(C8) edge (D2) (C9) edge (D3) (C7) edge (D1);
	\end{tikzpicture}
	\captionof{figure}{Fully redundant\\ (Alg.~B)}
	\label{fig:fully}
\end{minipage}
\end{center}
\end{figure}

We are now able to write down a saddle point problem in matrix-vector form that is equivalent to problem~\eqref{discreteVarProb}: find $(\underline{w},\underline{\lambda})$ such that
\[
	\left(
		\begin{array}{cccc}
			S& B^\top \\
			B			
		\end{array}
	\right)
	\left(
		\begin{array}{c}
			\underline{w}  \\
			\underline{\lambda}  \\
		\end{array}
	\right)
	=
	\left(
		\begin{array}{c}
			\underline{g}  \\
			0  \\
		\end{array}
	\right).
\]
Note that the matrices $A^{(k)}$ and $S^{(k)}$ are non-singular only if the
corresponding patch contributes to the Dirichlet boundary $\partial \Omega$. 
Since in general at least one of the matrices $S^{(k)}$ is singular,
the matrix $S$ is usually singular. 

The remedy for this problem are primal degrees of freedom which form a global but small linear system. As mentioned in the introduction, we consider
three options: Alg.~A, B, and C. Each of them corresponds to a different
choice of spaces $\widetilde{W}$ and $\widetilde{W}_\Delta$. 

\begin{itemize}
	\item \emph{Alg.~A (Vertex values):} The space $\widetilde{W}$
		is the subspace of functions where the vertex values agree, i.e.,
		\[
				\widetilde{W} :=
				\left\{ w\in W
						\;:\; 
						\begin{array}{l}
							w^{(k)}(\textbf x) = w^{(\ell,k)}(\textbf x) \\
							\mbox{for all common vertices $\textbf x$ of all
							$\Omega^{(k)}$ and all $\Omega^{(\ell)}$}
						\end{array}
							\right\}.
		\]	
		The space $\widetilde{W}_{\Delta}$ satisfies these conditions homogeneously, i.e.,
		$\widetilde{W}_{\Delta}:= \prod_{k=1}^K \widetilde{W}_{\Delta}^{(k)}$ and	
		\[
				\widetilde{W}_{\Delta}^{(k)} :=
				\left\{  w_e^{(k)}\in W_e^{(k)}
						\;:\; 
						\begin{array}{l}
							w^{(k)}(\textbf x) = w^{(\ell,k)}(\textbf x) = 0 \\
							\mbox{for all common vertices $\textbf x$ of 
							$\Omega^{(k)}$ and all $\Omega^{(\ell)}$}
						\end{array}
							\right\}.
		\]
	\item \emph{Alg.~B (Edge averages):} The space $\widetilde{W}$
		is the subspace of functions where the averages of the function values over
		the edges agree, i.e.,
		\[
				\widetilde{W} :=
				\left\{ w\in W
						\;:\; \int_{\Gamma^{(k,\ell)}} w^{(k)}(x)\,\mathrm dx
						= \int_{\Gamma^{(k,\ell)}} w^{(\ell,k)}(x)\,\mathrm dx
							\mbox{ for all edges }\Gamma^{(k,\ell)} \right\}.
		\]	
		The space $\widetilde{W}_{\Delta}$ satisfies these conditions homogeneously, i.e.,
		$\widetilde{W}_{\Delta}:= \prod_{k=1}^K \widetilde{W}_{\Delta}^{(k)}$ and
		\[
			\widetilde{W}_{\Delta}^{(k)} := \left\{ w_e^{(k)}\in W_e^{(k)} \;:\; \int_{\Gamma^{(k,\ell)}} w^{(k)}(x)\,\mathrm dx = \int_{\Gamma^{(k,\ell)}} w^{(k,\ell)}(x)\,\mathrm dx = 0 \mbox{ for all } \Gamma^{(k,\ell)} \right\}.
		\] 
	\item \emph{Alg.~C (Vertex values and edge averages):} We combine the constraints from
		both cases. So, the spaces $\widetilde{W}$ and $\widetilde{W}_{\Delta}^{(k)}$ and
		$\widetilde{W}_{\Delta}$ are the intersections of the corresponding spaces obtained by
		Alg.~A and B.
\end{itemize}

The primal space $\widetilde{W}_\Pi$ is the subspace of energy minimizing functions, cf. \cite{Pechstein:2013a}. This means that $\widetilde{W}_\Pi$ is $S$-orthogonal
to $\widetilde{W}_\Delta$, i.e., 
\begin{equation}\label{eq:wtildedef}
	\widetilde{W}_{\Pi} := \{ w\in \widetilde{W} \;:\; 
	s( w, q ) = 0
	\mbox{ for all } q \in \widetilde{W}_{\Delta} \}.
\end{equation} 

Let $\psi^{(1)},\ldots,\psi^{(N_\Pi)}$ be a basis of $\widetilde{W}_{\Pi}$. 
In practice, one
usually chooses a nodal basis, where the vertex values and/or the edge averages form
the nodal values. The matrix $\Psi$ represents the basis in terms of the basis
for the space $W$, i.e.,
\begin{equation}\label{eq:psidef}
	\Psi = \left(
		\begin{array}{c}
			\underline\psi^{(1)}\\\vdots\\\underline\psi^{(N_\Pi)}
		\end{array}
	\right).
\end{equation}

A representation for the spaces $\widetilde{W}_\Delta^{(k)}$ can be obtained by full-rank matrices $C^{(k)}$:
\[
	C^{(k)}\,\underline{w}_e^{(k)} = 0
	\quad\Leftrightarrow\quad
	 w_e^{(k)} \in \widetilde{W}_\Delta^{(k)}.
\]
The block-diagonal collection of the matrices $C^{(k)}$ gives the matrix 
\[
	C := 
	\begin{pmatrix}
		C^{(1)} &  & \\
		& \ddots & \\
		& & C^{(K)}
	\end{pmatrix}.
\]

As in \cite{MandelDohrmannTezaur:2005a}, the following problem
is an equivalent reformulation of \eqref{discreteVarProb}: find
$(\underline w_\Delta, \underline \mu, \underline w_\Pi, \underline \lambda)$ such that
\[
	\left(
		\begin{array}{cccc}
			S & C^\top &                  & B^\top           \\
			C &        &                  &                  \\
			  &        & \Psi^\top S \Psi & \Psi^\top B^\top \\
			B &        & B \Psi           &                  \\
		\end{array}
	\right)
	\left(
		\begin{array}{c}
			\underline{w}_\Delta  \\
			\underline{\mu}  \\
			\underline{w}_\Pi  \\
			\underline{\lambda}  \\
		\end{array}
	\right)
	=
	\left(
		\begin{array}{c}
			\underline{g}  \\
			0  \\
			\Psi^\top \underline{g}  \\
			0  \\
		\end{array}
	\right).
\]
The solution for the original problem is obtained by
$\underline{w} = \underline{w}_\Delta + \Psi \underline{w}_\Pi$.
Following the DP approach, this system can be reduced to the subsequent problem for the Lagrange multipliers \underline{$\lambda$}:
\begin{equation}
	\label{IETIProblem}
	F \; \underline{\lambda} = \underline{d},
\end{equation}
where
\begin{align}
	\label{eq:IETI-matrix}
	F :=  \underbrace{\left(
			\begin{array}{ccc}
			B & 0 & B \Psi
			\end{array}
			\right)
			\left(
			\begin{array}{ccc}
			S & C^\top &  \\ C &   &   \\   &   & \Psi^\top S\Psi
			\end{array}
			\right)^{-1}
			\left(
			\begin{array}{c}
			 I  \\ 0 \\ \Psi^\top 
			\end{array}
			\right)}_{\displaystyle F_0:=}
			B^\top
			\quad\mbox{and}\quad
			\underline{d} :=  F_0\, \underline{g}.
\end{align}

The system~\eqref{IETIProblem} is solved with a preconditioned conjugate gradient (PCG) solver. In order to obtain an effective solver, we need a proper preconditioner. We use the scaled Dirichlet preconditioner $M_{\mathrm{sD}}$, defined by
\[
		M_{\mathrm{sD}} := B D^{-1} S  D^{-1} B^\top ,
\]
where $D \in \mathbb{R}^{N_{\Gamma}\times N_{\Gamma}}$ is a diagonal matrix
defined based on the principle of multiplicity scaling: Each
coefficient $d_{i,i}$ of $D$ is assigned the number of constraints for the corresponding degree of freedom plus one i.e., 
\[
	d_{i,i} := 1 +
		 \sum_{j=1}^{N_\Gamma} b_{j,i}^2,
\]
where $b_{i,j}$ are the coefficients of the matrix $B$.

The dG IETI-DP method requires the execution of the following steps. 
\begin{itemize}
	\item Compute the right-hand sides $\underline g_e^{(k)}$ according to \eqref{eq:s:def}.
	\item Compute the basis $\Psi$ for the primal space in accordance with~\eqref{eq:wtildedef} and~\eqref{eq:psidef}.
	\item Compute the primal Schur complement $S_\Pi := \Psi^{\top} S \Psi$.
	\item Solve \eqref{IETIProblem} for $\underline \lambda$ using a PCG solver.
	This requires the computation of the residual and the application
	of the preconditioner.
		
		The computation of the residual $\underline{\widehat{w}}:= F \underline{\lambda} - \underline{d}$ requires the following steps:
		\begin{enumerate}
			\item Compute $\underline{\widehat q}
			= ((\underline{\widehat q}_e^{(1)})^\top \cdots (\underline{\widehat q}_e^{(K)})^\top)^\top
				:= B^\top \underline \lambda - \underline g$.
			\item Solve the linear system
				\begin{equation}\label{eq:algo1}
						\left(
							\begin{array}{cc}
								S^{(k)} & (C^{(k)})^\top \\ C^{(k)}
							\end{array}
						\right)
						\left(
							\begin{array}{c}
								\underline{\widehat w}_\Delta^{(k)} \\ \underline{\widehat \mu}^{(k)}
							\end{array}
						\right)
						=
						\left(
							\begin{array}{c}
								\underline{\widehat q}_e^{(k)} \\ 0
							\end{array}
						\right)
				\end{equation}
				for all $k=1,\ldots,K$. For Alg.~A and C, one usually reduces the local systems by eliminating the degrees of freedom corresponding to the vertex values and the corresponding Lagrange multipliers. 
			\item Solve the (usually small) global linear system
				\begin{equation}\label{eq:algo2}
						S_\Pi \underline{\widehat w}_\Pi = \Psi^\top \underline{\widehat q}.
				\end{equation}
			\item Compute the residual
				\begin{equation}\label{eq:algo3}
						\underline{\widehat w} := B \left(
							\begin{array}{ccc}
								\underline{\widehat w}_\Delta^{(1)} \\ \vdots \\ \underline{\widehat w}_\Delta^{(K)}
							\end{array}
						\right)
							+ B \Psi  \underline {\widehat w}_\Pi.
				\end{equation}
		\end{enumerate}
		The application of the preconditioner to the residual, i.e., the computation of
		$M_{\mathrm{sD}}\,	\underline{\widehat w}$, only requires
		matrix-vector multiplications.
	\item The vector $\underline w$, representing the solution on the skeleton, is
	computed analogously to~\eqref{eq:algo1}, \eqref{eq:algo2} and~\eqref{eq:algo3}
	based on $q=((q^{(1)})^\top) \cdots (q^{(K)})^\top)^\top:=B^\top \underline{\lambda}$. Finally, the solution vector $\underline{u}$ is derived by~\eqref{eq:inner:def}.
\end{itemize}

\section{The condition number estimate}
\label{sec:4}
In this section, we prove the following condition number estimates.

\begin{theorem}\label{thrm:fin}
	Provided that the IETI-DP solver is set up as outlined in the previous
	sections, the condition number of the preconditioned system satisfies
	\[
		\kappa(M_{\mathrm{sD}} F) \le C\, p \; 
		\left(1+\log p+\max_{k=1,\ldots,K} \log\frac{H_{k}}{h_{k}}\right)^2
	\]
	for Alg.~A and C and
	\[
		\kappa(M_{\mathrm{sD}} F) \le C\,  
		\delta\, p
		\left(
				\max_{k=1,\ldots,K} \max_{\ell \in \mathcal N_{\Gamma}(k)}
				\frac{h_{k}}{h_{\ell}}
		\right)
		\left(1+\log p+\max_{k=1,\ldots,K} \log\frac{H_{k}}{h_{k}}\right)^2
	\]
	for Alg.~B,
	where in both cases the constant $C$ only depends on the constants from the
	Assumptions~\ref{ass:nabla}, \ref{ass:neighbors}, and~\ref{ass:quasiuniform}.
\end{theorem}

If some Lagrange multipliers are redundant to each other or to
primal constraints, the matrices $M_{\mathrm{sD}}$ and $F$ are singular. Standard
iteration schemes live in the corresponding factor space. Here and in what
follows, the condition number is estimated for the restriction 
to the factor space, cf.~\cite[Remark~23]{MandelDohrmannTezaur:2005a}.

\begin{notation}
	We use the notation $a \lesssim b$ if there is a constant
	$c>0$ that only depends on the constants from the
	Assumptions~\ref{ass:nabla}, \ref{ass:neighbors}, and~\ref{ass:quasiuniform}
	such that $a\le cb$. Moreover, we write $a\eqsim b$ if $a\lesssim b\lesssim a$.
	
	When it is clear from the context, we do not denote the restriction of
	a function to an interface explicitly, so we write for example
	$\|u^{(k)}\|_{L_2(\Gamma^{(k,\ell)})}$
	instead of
	$\|u^{(k)}|_{\Gamma^{(k,\ell)}}\|_{L_2(\Gamma^{(k,\ell)})}$.
\end{notation}

Before we dive into the condition number estimate, we observe that the bilinear forms $a_e^{(k)}(\cdot, \cdot)$ and 
$d_e^{(k)}(\cdot, \cdot)$ are patchwise equivalent.
\begin{lemma}\label{lem:patchwise equ}
	For every patch $\Omega^{(k)}$,
	\[
		a_e^{(k)}(u_e^{(k)}, u_e^{(k)}) \lesssim d_e^{(k)}(u_e^{(k)}, u_e^{(k)}) \lesssim a_e^{(k)}(u_e^{(k)}, u_e^{(k)}) 
	\]
	holds for all $u_e^{(k)} \in V_e^{(k)}$.
\end{lemma}
\begin{proof}
	Completely analogous to the proof of \cite[Theorem 8]{Takacs:2019b}.
\end{proof}

Analogous to~\cite[Lemma~4.16]{SchneckenleitnerTakacs:2019},
the following lemma allows to estimate the action of the matrix $B^\top_D B$.
\begin{lemma}\label{lem:bbt}
	Let $u=(u_e^{(1)},\cdots,u_e^{(K)})=((u^{(1)},(u^{(1,\ell)})_{\ell\in \mathcal N_\Gamma(1)}),\ldots)\in \widetilde W$ with coefficient vector $\underline u$
	and let $w=(w_e^{(1)},\cdots,w_e^{(K)})=((w^{(1)},(w^{(1,\ell)})_{\ell\in \mathcal N_\Gamma(1)}),\ldots)\in \widetilde W$ with coefficient vector $\underline w$
	be such that $\underline w = B_D^\top B \underline u$.
	Then, we have for each patch $\Omega^{(k)}$ and each edge $\Gamma^{(k,\ell)}$
	connecting the vertices ${\normalfont \textbf x}^{(k,\ell,1)}$ and ${\normalfont \textbf x}^{(k,\ell,2)}$
	\[
		\begin{aligned}
			& \|w^{(k)} - w^{(k,\ell)} \|_{L_{2}(\Gamma^{(k,\ell)})}^2 \\
			& \qquad \lesssim
			    \|u^{(k)} - u^{(k,\ell)}\|_{L^{2}(\Gamma^{(k,\ell)})}^2
			  + \|u^{(\ell)}- u^{(\ell,k)} \|_{L^{2}(\Gamma^{(k,\ell)})}^2 
			  + \sum_{i=1}^2 
			  \left(
					  	\frac{h_{k}}{p} \Delta^{(k,\ell,i)} 
			  			+ \frac{h_{\ell}}{p} \Delta^{(\ell,k,i)}
			  \right),\\
		& |w^{(k)}|_{H^{1/2}(\Gamma^{(k,\ell)})}^2  \lesssim
		|u^{(k)}|_{H^{1/2}(\Gamma^{(k,\ell)})}^2 +
		|u^{(\ell,k)}|_{H^{1/2}(\Gamma^{(k,\ell)})}^2 + \sum_{i=1}^2 \Delta^{(k,\ell,i)},
	\end{aligned}
	\]
	and 
	\[
		|w^{(k)}|_{L_\infty^0(\Gamma^{(k,\ell)})}^2
		\lesssim
		|u^{(k)}|_{L_\infty^0(\Gamma^{(k,\ell)})}^2 +
		|u^{(\ell,k)}|_{L_\infty^0(\Gamma^{(k,\ell)})}^2  
		+ \sum_{i=1}^2 \Delta^{(k,\ell,i)},
	\]
	where
	\begin{equation}\label{eq:def:delta}
		\Delta^{(k,\ell, i)}:=
		\begin{cases}
		0	&\mbox{for Alg.~A and C}\\
			\sum_{j \in \mathcal{P}(\normalfont \textbf x^{(k,\ell,i)}) \cap \mathcal{N}_\Gamma(k)}
			|u^{(k)}({\normalfont \textbf x}^{(k,\ell,i)})-u^{(j,k)}({\normalfont \textbf x}^{(k,\ell,i)})|^2 &\mbox{for Alg.~B}.
	\end{cases}
	\end{equation}
\end{lemma}
\begin{proof}
	We start by recalling how the scaling matrix $D$ looks like. Remember $D$ is a
	diagonal matrix. We denote the coefficients by $d_{i,i}$ for $i=1,\ldots,N_\Gamma$.
	The entries $d_{i,i}$ are defined to be one plus
	the number of Lagrange multipliers of the corresponding degree of freedom $i$.
	Note that we have $d_{i,i}=2$ unless the corresponding degree of freedom is
	a corner degree of freedom.
	In that case, $d_{i,i}$ takes a value, which is bounded from below
	by $1$ and bounded from above due to Assumption~\ref{ass:neighbors}. Thus,
	$d_{i,i}\eqsim 1$ in any case.
		
	We will start with Alg.~A and C:
	Simple calculations reveal for Alg.~A and C that
	\[\begin{aligned}
	w^{(k)}|_{\Gamma^{(k,\ell)}} &= \frac{1}{2}
	\left(
	u^{(k)}|_{\Gamma^{(k,\ell)}} -
	u^{(\ell,k)}  
	- \sum_{i=1}^2 \theta^{(k,\ell,i)}
	\big( u^{(k)}(\textbf x^{(k,\ell,i)}) - u^{(\ell,k)}(\textbf x^{(k,\ell,i)})\big) \right), \\ 
	w^{(k,\ell)} &=
	\frac{1}{2} 
	\left( 
	u^{(k, \ell)} - 
	u^{(\ell)}|_{\Gamma^{(k,\ell)}}
	- \sum_{i=1}^2 \theta^{(\ell,k,i)}
	\big( u^{(k,\ell)}(\textbf x^{(k,\ell,i)}) - u^{(\ell)}(\textbf x^{(k,\ell,i)})\big) 
	\right).
	\end{aligned}\]
	where $\theta^{(k,\ell,i)}$ is the basis function in $\Phi^{(k)}$
	such that $\theta^{(k,\ell,i)}(\textbf x^{(k,\ell,i)})=1$. 
	$\theta^{(\ell, k, i)}$ is defined analogously.
	Since $u$ satisfies the primal constraints, we have
	$u^{(k)}(\textbf x^{(k,\ell,i)}) = u^{(\ell,k)}(\textbf x^{(k,\ell,i)})$ and
	$u^{(k,\ell)}(\textbf x^{(k,\ell,i)}) = u^{(\ell)}(\textbf x^{(k,\ell,i)})$
	and thus
	\[
	\begin{aligned}
	w^{(k)}|_{\Gamma^{(k,\ell)}} &= \frac{1}{2}
	(
	u^{(k)}|_{\Gamma^{(k,\ell)}} - u^{(\ell,k)}), \qquad
	w^{(k,\ell)} &= \frac{1}{2} 
	(u^{(k,\ell)} - u^{(\ell)}|_{\Gamma^{(k,\ell)})}).
	\end{aligned}
	\] \\
	Therefore, we obtain by using the triangle inequality 
	\[
	\begin{aligned}
		\|w^{(k)}-w^{(k,\ell)}\|_{L_{2}(\Gamma^{(k,\ell)})}^2
		&\lesssim 
		\|u^{(k)} - u^{(k,\ell)}\|_{L_{2}(\Gamma^{(k,\ell)})}^2 + 
		\| u^{(\ell)} - u^{(\ell,k)}\|_{L_{2}(\Gamma^{(k,\ell)})}^2 
	\end{aligned}
	\]
	for the $L_2$-norm. For the $H^{1/2}$-seminorm, we get using the triangle inequality that
	\[
		|w^{(k)}|_{H^{1/2}(\Gamma^{(k,\ell)})}^2
		\lesssim  |u^{(k)}|_{H^{1/2}(\Gamma^{(k,\ell)})}^2 + |u^{(\ell,k)}|_{H^{1/2}(\Gamma^{(k,\ell)})}^2 .
	\]
	Analogously, we obtain using the triangle inequality 
	\[
	 |w^{(k)}|_{L_\infty^0(\Gamma^{(k,\ell)})}^2 
	\lesssim
	|u^{(k)}|_{L_\infty^0(\Gamma^{(k,\ell)})}^2 +
	|u^{(\ell,k)}|_{L_\infty^0(\Gamma^{(k,\ell)})}^2,
	\] 
	which finishes the proof for the Alg.~A and C.	
	
	For Alg.~B we have that $w^{(k)}$ on $\Gamma^{(k,\ell)}$ can be expanded as 	\[
	\begin{aligned}
	w^{(k)}|_{\Gamma^{(k,\ell)}} &= 
	\frac{1}{2} \left(
	u^{(k)}|_{\Gamma^{(k,\ell)}} -
	u^{(\ell,k)}  
	- \sum_{i=1}^2 \theta^{(k,\ell,i)}
	\big( u^{(k)}(\textbf x^{(k,\ell,i)}) - u^{(\ell,k)}(\textbf x^{(k,\ell,i)})\big) \right) \\ & \quad 
	+ \sum_{i=1}^{2}\frac{1}{d_{n_i,n_i}}\theta^{(k,\ell,i)}\sum_{j\in \mathcal{P}(\textbf x^{(k,\ell,i)}) \cap \mathcal{N}_\Gamma(k)}
	\left(
	u^{(k)}(\mathbf{x}^{(k,\ell,i)}) - u^{(j,k)}(\mathbf{x}^{(k,\ell,i)})
	\right)
	, \\ 
	w^{(k,\ell)} &=
	\frac{1}{2} 
	\left( 
	u^{(k, \ell)} - 
	u^{(\ell)}|_{\Gamma^{(k,\ell)}}
	- \sum_{i=1}^2 \theta^{(\ell,k,i)}
	\big( u^{(k,\ell)}(\textbf x^{(k,\ell,i)}) - u^{(\ell)}(\textbf x^{(k,\ell,i)})\big) \right) \\ & \quad 
	+ \sum_{i=1}^{2}\frac{1}{d_{m_i,m_i}}\theta^{(\ell,k,i)}\sum_{j\in (\mathcal{P}(\textbf x^{(k,\ell,i)}) \cap \mathcal{N}_\Gamma(\ell)) \cup \{\ell\} }
	\left(
	u^{(k,\ell)}(\mathbf{x}^{(k,\ell,i)}) - u^{(j,\ell)}(\mathbf{x}^{(k,\ell,i)})
	\right)
	, \\
	\end{aligned}
	\]
	where $d_{n_i,n_i}$ and $d_{m_i,m_i}$ denote the entries of the matrix $D$ for the corresponding dofs and we use $u^{(\ell,\ell)} := u^{(\ell)}$.  
	Note that $\theta^{(k,\ell,i)}$ behaves like $\text{max}\{0,1-|x-\textbf x^{(k,\ell,i)}| / h^{(k)} \}^{p}$. Hence,
	\[
	\begin{aligned}
		\| \theta^{(k,\ell,i)} \|_{L_2(\Gamma^{(k,\ell)})}^2 &\eqsim \frac{h_{k}}{p}, \\
		|\theta^{(k,\ell,i)}|_{H^{1/2}(\Gamma^{(k,\ell)})}^2 &\leq |\theta^{(k,\ell,i)}|_{H^1(\Gamma^{(k,\ell)})} \|\theta^{(k,\ell,i)} \|_{L_2(\Gamma^{(k,\ell)})} \eqsim 1,  \text{ and } \\
		\| \theta^{(k,\ell,i)} \|_{L_\infty(\Gamma^{(k,\ell)})}^2 &= 1.
	\end{aligned}
	\]
	An application of the triangle inequality for the corner expressions 
	yields the stated result for Alg.~B. 
\end{proof}

The following lemma allows to estimate the $H^{1/2}$-seminorm on the
artificial edges.
\begin{lemma}\label{lem:hhalf:on:artif:edge}
	For any two patches $\Omega^{(k)}$ and $\Omega^{(\ell)}$ that share an edge $\Gamma^{(k,\ell)}$, the estimate
	\[
		|u^{(k,\ell)} |_{H^{1/2}{(\Gamma^{(k,\ell)})}}^2 \lesssim 
		|u^{(k)}|_{H^{1/2}(\Gamma^{(k,\ell)})}^2
		+ \frac{p^2}{h_{k\ell}} \| u^{(k,\ell)} - u^{(k)} \|_{L_2(\Gamma^{(k,\ell)})}^2 
	\]
	holds for all $u_e^{(k)} = \left( u^{(k)}, \left( u^{(k,\ell)}\right)_{\ell \in \mathcal{N}_\Gamma(k)} \right) \in V_e^{(k)}$.
\end{lemma}
\begin{proof}
		Let $c\in \mathbb R$ be arbitrary but fixed.
		Using the triangle inequality, we have
		\[
			\|u^{(k,\ell)} - c \|_{H^{1/2}{(\Gamma^{(k,\ell)})}}^2
			\lesssim
			\|u^{(k)} - c \|_{H^{1/2}{(\Gamma^{(k,\ell)})}}^2
			+
			\|(u^{(k)} - c) - (u^{(k,\ell)} - c) \|_{H^{1/2}{(\Gamma^{(k,\ell)})}}^2.
		\]
		Using \cite[Theorem 5.2, eq. (3)]{AdamsFournier:2003}
		and the triangle inequality, we obtain further
		\[
		\begin{aligned}
			& \|u^{(k,\ell)} - c \|_{H^{1/2}{(\Gamma^{(k,\ell)})}}^2 \\
			& \quad \lesssim
			\|u^{(k)} - c \|_{H^{1/2}{(\Gamma^{(k,\ell)})}}^2
			+
			\|u^{(k)} - u^{(k,\ell)} \|_{L_2{(\Gamma^{(k,\ell)})}}
			\|(u^{(k)} - c) - (u^{(k,\ell)} - c) \|_{H^{1}{(\Gamma^{(k,\ell)})}}\\
			& \quad \le
			\|u^{(k)} - c \|_{H^{1/2}{(\Gamma^{(k,\ell)})}}^2
			\\ & \qquad\qquad +
			\|u^{(k)} - u^{(k,\ell)} \|_{L_2{(\Gamma^{(k,\ell)})}}
			\left(\|u^{(k)} - c\|_{H^{1}{(\Gamma^{(k,\ell)})}} + \|u^{(k,\ell)} - c \|_{H^{1}{(\Gamma^{(k,\ell)})}}\right).
		\end{aligned}
		\]
		Using the equivalence of the norms on the parameter domain and the physical domain, cf. \cite[Lemma 4.13]{SchneckenleitnerTakacs:2019}, and an inverse inequality, cf. \cite[Lemma 4.3]{Schwab}
		and using $h_{k\ell}=\min\{h_{k}, h_{\ell}\}$, we obtain further
		\[
		\begin{aligned}
			& \|u^{(k,\ell)} - c \|_{H^{1/2}{(\Gamma^{(k,\ell)})}}^2 
			 \le
			\|u^{(k)} - c\|_{H^{1/2}{(\Gamma^{(k,\ell)})}}^2
			\\ & \qquad + 
			\frac{C\; p}{h_{k\ell}^{1/2}}
			\|u^{(k)} - u^{(k,\ell)} \|_{L_2{(\Gamma^{(k,\ell)})}}
			\left(
			\|u^{(k)} - c \|_{H^{1/2}{(\Gamma^{(k,\ell)})}} \right. \left.
			+ \|u^{(k,\ell)} - c \|_{H^{1/2}{(\Gamma^{(k,\ell)})}}
			\right),
		\end{aligned}
		\]
		where $C>0$ is an appropriately chosen constant (that only depends
		on the constant from Assumption~\ref{ass:nabla}). Using
		$ab \le a^2 + b^2/4$ and using $(a+b)^2\le 2(a^2+b^2)$, we have
		\[
		\begin{aligned}
			&\|u^{(k,\ell)} - c \|_{H^{1/2}{(\Gamma^{(k,\ell)})}}^2\\
			& \quad \le
			\|u^{(k)} - c\|_{H^{1/2}{(\Gamma^{(k,\ell)})}}^2
			+
			\frac{C^2\; p^2}{h_{k\ell}}
			\|u^{(k)} - u^{(k,\ell)} \|_{L_2{(\Gamma^{(k,\ell)})}}^2
			\\ & \hspace{5cm} +
			\frac{1}{4}
			\left(\|u^{(k)} - c\|_{H^{1/2}{(\Gamma^{(k,\ell)})}} + \|u^{(k,\ell)} - c\|_{H^{1/2}{(\Gamma^{(k,\ell)})}}\right)^2\\
			&\quad\le
			\frac{3}{2}\|u^{(k)} - c\|_{H^{1/2}{(\Gamma^{(k,\ell)})}}^2
			+
			\frac{C^2\; p^2}{h_{k\ell}}
			\|u^{(k)} - u^{(k,\ell)} \|_{L_2{(\Gamma^{(k,\ell)})}}^2
			+
			\frac{1}{2}
			\|u^{(k,\ell)} - c\|_{H^{1/2}{(\Gamma^{(k,\ell)})}}^2.
		\end{aligned}
		\]
		By subtracting $\frac{1}{2}\|u^{(k,\ell)} -c\|_{H^{1/2}{(\Gamma^{(k,\ell)})}}^2$,
		we immediately obtain 
		\[
			|u^{(k,\ell)} |_{H^{1/2}{(\Gamma^{(k,\ell)})}}^2
			\le
			\|u^{(k,\ell)} - c \|_{H^{1/2}{(\Gamma^{(k,\ell)})}}^2
			\lesssim 
			\|u^{(k)} - c\|_{H^{1/2}{(\Gamma^{(k,\ell)})}}^2
			+ 
			\frac{p^2}{h_{k\ell}}
			\|u^{(k)} - u^{(k,\ell)} \|_{L_2{(\Gamma^{(k,\ell)})}}^2
		\]	
		for all $u_e^{(k)} \in V_e^{(k)}$. Since this holds for all $c\in \mathbb R$, the
		Poincar{\'e} inequality yields the desired result.
\end{proof}

The next step is to show that a similar estimate holds for the seminorm 
\[
	|v|_{L_\infty^0(T)}^2 := \inf_{c\in \mathbb R} \| v - c\|_{L_\infty(T)}.
\]
Before we can prove that result, we need the following auxiliary result. 
\begin{lemma}
	\label{lem:max estimate}
	The estimate 
	\[
		\| u \|_{L_\infty(0,1)}^2 \leq \sqrt{2} \| u \|_{L_2(0,1)} \| u \|_{H^1(0,1)}
	\]
	holds for all $u \in H^1(0,1)$.
\end{lemma}
\begin{proof}
	Since $u$ is continuous, $u$ takes its maximum for some $z \in [0,1]$. Then, by the fundamental theorem of differential and integral calculus we can write
	\[
		|u(z)|^2 = -\int_{z}^{t}u(s)u'(s) \; \textrm{d}s + | u(t) |^2.
	\]
	Next, we integrate over the unit interval and use the Cauchy-Schwarz inequality to obtain
	\[
	\begin{aligned}
		|u(z)|^2 
		&= -\int_{0}^{1} \int_{z}^{t}u(s)u'(s) \; \textrm{d}s + | u(t) |^2 \; \textrm{d}t \leq \int_{0}^{1} \| u \|_{L_2(z,t)} \| u' \|_{L_2(z,t)} \; \textrm{d}t + \int_{0}^{1}|u(t)|^2 \; \textrm{d}t \\
		&\leq \int_{0}^{1} \| u \|_{L_2(0,1)} \| u' \|_{L_2(0,1)} \; \textrm{d}t + \| u \|_{L_2(0,1)}^2 = \| u \|_{L_2(0,1)} 
		\left( \| u' \|_{L_2(0,1)} + \| u \|_{L_2(0,1)} \right) \\
		&\leq \sqrt{2} \| u \|_{L_2(0,1)} \| u \|_{H^1(0,1)},
	\end{aligned}
	\]
	which was to show.
\end{proof}
The next lemma allows to estimate the $L^0_\infty$-seminorm similar to
Lemma~\ref{lem:hhalf:on:artif:edge}.
\begin{lemma}\label{lem:linf:on:artif:edge}  
	For any two patches $\Omega^{(k)}$ and $\Omega^{(\ell)}$ that share an edge $\Gamma^{(k,\ell)}$, the inequality 
	\[
		|u^{(k,\ell)}|_{L_\infty^0(\Gamma^{(k,\ell)})}^2 \lesssim 
		| u^{(k)} |_{L_\infty^0(\Gamma^{(k,\ell)})}^2 +
		| u^{(k)} |_{H^{1/2}(\Gamma^{(k,\ell)})}^2 + 
		\frac{p^2}{h_{k\ell}}\| u^{(k)} - u^{(k,\ell)} \|_{L_2(\Gamma^{(k,\ell)})}^2
	\]
	holds for all $u_e^{(k)} = \left( u^{(k)}, \left( u^{(k,\ell)} \right)_{\ell \in \mathcal{N}_\Gamma(k)} \right) \in V_e^{(k)}$.
\end{lemma}
\begin{proof}
	Using the triangle inequality, we obtain 
	\[
		| u^{(k,\ell)} |_{L_\infty^0(\Gamma^{(k,\ell)})}^2
		\lesssim | u^{(k)} |_{L_\infty^0(\Gamma^{(k,\ell)})}^2 + 
		| u^{(k)} - u^{(k,\ell)} |_{L_\infty^0(\Gamma^{(k,\ell)})}^2.
	\]
	As a next step, we apply Lemma \ref{lem:max estimate} to the difference $| u^{(k)} - u^{(k,\ell)} |_{L_\infty^0(\Gamma^{(k,\ell)})}^2$.
	Since the $L_\infty$-norm does not change if we are on the physical or the parameter domain, we can apply Lemma \ref{lem:max estimate} and subsequently utilize the equivalence of the norms on the parameter and the physical domain, cf. \cite[Lemma 4.13]{SchneckenleitnerTakacs:2019} to get
	\[
	\begin{aligned}
		| u^{(k)} - u^{(k,\ell)} |_{L_\infty^0(\Gamma^{(k,\ell)})}^2
		&\lesssim 
		\| u^{(k)} - u^{(k,\ell)} \|_{L_2(\Gamma^{(k,\ell)})}
		\| u^{(k)} - u^{(k,\ell)} \|_{H^1(\Gamma^{(k,\ell)})}.
	\end{aligned}
	\]
	The triangle inequality allows to estimate
	\[
		| u^{(k)} - u^{(k,\ell)} |_{L_\infty^0(\Gamma^{(k,\ell)})}^2 
		\lesssim
		\| u^{(k)} - u^{(k,\ell)} \|_{L_2(\Gamma^{(k,\ell)})}
		\left(\| u^{(k)}-c \|_{H^1(\Gamma^{(k,\ell)})} + \| u^{(k,\ell)}-c \|_{H^1(\Gamma^{(k,\ell)})}\right)
	\]
		for all $c\in \mathbb R$.
	The equivalence of the norms on the physical domain and the parameter domain, cf. \cite[Lemma 4.13]{SchneckenleitnerTakacs:2019} and an inverse estimate, cf. \cite[Lemma 4.3]{SchneckenleitnerTakacs:2019}, give
	\[
	\begin{aligned}
		&| u^{(k)} - u^{(k,\ell)} |_{L_\infty^0(\Gamma^{(k,\ell)})}^2 \\
		& \quad \lesssim
		\frac{p}{h_{k\ell}^{1/2}}\| u^{(k)} - u^{(k,\ell)} \|_{L_2(\Gamma^{(k,\ell)})}
		\left(\| u^{(k)}-c \|_{H^{1/2}(\Gamma^{(k,\ell)})} + \| u^{(k,\ell)}-c \|_{H^{1/2}(\Gamma^{(k,\ell)})}\right).
	\end{aligned}
	\]
	We use $ab \lesssim a^2 + b^2$ and $(a+b)^2 \lesssim a^2 + b^2$ to get
	\[
	\begin{aligned}
		& | u^{(k)} - u^{(k,\ell)} |_{L_\infty^0(\Gamma^{(k,\ell)})}^2 \\
		&\quad \lesssim
		\frac{p^2}{h_{k\ell}}\| u^{(k)} - u^{(k,\ell)} \|_{L_2(\Gamma^{(k,\ell)})}^2
		+ \| u^{(k)}-c \|_{H^{1/2}(\Gamma^{(k,\ell)})}^2 + \|u^{(k,\ell)}-c \|_{H^{1/2}(\Gamma^{(k,\ell)})}^2.
	\end{aligned}
	\]
	Since this holds for all $c\in \mathbb R$, 
	the Poincar{\'e} inequality gives further
	\[
	\begin{aligned}
		| u^{(k)} - u^{(k,\ell)} |_{L_\infty^0(\Gamma^{(k,\ell)})}^2 
		\lesssim
		\frac{p^2}{h_{k\ell}}\| u^{(k)} - u^{(k,\ell)} \|_{L_2(\Gamma^{(k,\ell)})}^2
		+ | u^{(k)} |_{H^{1/2}(\Gamma^{(k,\ell)})}^2 
		+ |u^{(k,\ell)} |_{H^{1/2}(\Gamma^{(k,\ell)})}^2.
	\end{aligned}
	\]
	Lemma \ref{lem:hhalf:on:artif:edge} yields the final result. 
\end{proof}

\cite[Lemma 4.17]{SchneckenleitnerTakacs:2019} states that the term $\Delta^{(k,\ell,i)}$ can be estimated by expressions which only involve patches sharing an edge.  
Now we prove a variant of \cite[Lemma 4.18]{SchneckenleitnerTakacs:2019}
that fits our needs.
\begin{lemma}
	\label{lem:corner estimate}
	For any two patches $\Omega^{(k)}$ and $\Omega^{(\ell)}$, sharing an edge $\Gamma^{(k,\ell)}$ which connects the vertices $\mathbf x^{(k,\ell,1)}$ and $\mathbf x^{(k,\ell, 2)}$ and 
	$
	\int_{\Gamma^{(k,\ell)}}u^{(k)}(s) - u^{(\ell,k)}(s) \; \mathrm{d}s = 0
	$ 
	and $i = 1,2$, we have 
	\[
		\Delta^{(k,\ell,i)}
		\lesssim
		\Lambda
		\left(
			|\mathcal{H}_h^{(k)} u^{(k)}|_{H^1(\Omega^{(k)})}^2 + 
			|\mathcal{H}_h^{(\ell)} u^{(\ell)}|_{H^1(\Omega^{(\ell)})}^2
		\right)
		+ \frac{p^2}{h_{k\ell}} 
					\| u^{(\ell)} - u^{(\ell, k)} \|_{L_2(\Gamma^{(k,\ell)})}^2 
	\]
	for all $u = \left(
		u_e^{(1)}, \dots, u_e^{(K)} 
	\right) \in \widetilde{W}$, where 
	$
		\Lambda = 1+\mathrm{log} \; p + \mathrm{max}_{j = 1,\dots,K} \mathrm{log} \frac{H_{j}} {h_{j}}
	$
	and $\Delta^{(k,\ell,i)}$ is as in~\eqref{eq:def:delta}.
\end{lemma}
\begin{proof}
	By the triangle inequality we have
	\begin{equation}\label{eq:proof:lem:corner:estimate:1}
	\begin{aligned}
		\Delta^{(k,\ell,i)} &=
		| u^{(k)}(\textbf x^{(k,\ell,i)}) - u^{(\ell,k)}(\textbf x^{(k,\ell,i)}) |^2 \\
		&\lesssim
		| u^{(k)}(\textbf x^{(k,\ell,i)}) - u^{(\ell)}(\mathbf{x}^{(k,\ell,i)}) |^2  
		+| u^{(\ell)}(\textbf x^{(k,\ell,i)}) - u^{(\ell,k)}(\mathbf{x}^{(k,\ell,i)}) |^2.
	\end{aligned}
	\end{equation}
	Let 
	\[
	\begin{aligned}
		\rho :=
		\frac{1}{|\Gamma^{(k,\ell)}|}
		\left(
		\int_{\Gamma^{(k,\ell)}} u^{(k)} \; \mathrm{d}s - \int_{\Gamma^{(k,\ell)}} u^{(\ell)} \; \mathrm{d}s \right) = 
		\frac{1}{|\Gamma^{(k,\ell)}|}
		\left(
		\int_{\Gamma^{(k,\ell)}} u^{(\ell,k)} \; \mathrm{d}s - \int_{\Gamma^{(k,\ell)}} u^{(\ell)} \; \mathrm{d}s
		\right)
	\end{aligned}
	\]
	and observe that the Cauchy-Schwarz inequality yields 
	\[
		\rho^2 \leq \frac{1}{ | \Gamma^{(k,\ell)} |} \| u^{(\ell, k)} - u^{(\ell)} \|^2_{L_2(\Gamma^{(k,\ell)})} .
	\]
	We further observe that due to Assumption~\ref{ass:nabla}, $|\Gamma^{(k,\ell)}| \eqsim \min \{ H_k, H_\ell \}$.
	The triangle inequality yields
	\[
	\begin{aligned}
		| u^{(k)}(\textbf x^{(k,\ell,i)}) - u^{(\ell)}(\mathbf{x}^{(k,\ell,i)}) |^2 
		&\lesssim
		| \left(u^{(k)}-\rho\right)(\textbf x^{(k,\ell,i)}) - u^{(\ell)}(\mathbf{x}^{(k,\ell,i)}) |^2 + |\rho|^2.
	\end{aligned}
	\]
	Using \cite[Lemma 4.18]{SchneckenleitnerTakacs:2019}, we obtain further
	\begin{equation}\label{eq:proof:lem:corner:estimate:2}
	\begin{aligned}
		& | u^{(k)}(\textbf x^{(k,\ell,i)}) - u^{(\ell)}(\mathbf{x}^{(k,\ell,i)}) |^2 \\
		& \qquad \lesssim
		\Lambda^{} (|\mathcal{H}_h^{(k)} u^{(k)}|_{H^1(\Omega^{(k)})}^2 + |\mathcal{H}_h^{(\ell)} u^{(\ell)}|_{H^1(\Omega^{(\ell)})}^2 )
		+ \frac{1}{\min \{H_k, H_\ell \} }\| u^{(\ell, k)} - u^{(\ell)} \|_{L_2(\Gamma^{(k,\ell)})}^2  
		\\
		& \qquad \le \Lambda^{} (|\mathcal{H}_h^{(k)} u^{(k)}|_{H^1(\Omega^{(k)})}^2 + |\mathcal{H}_h^{(\ell)} u^{(\ell)}|_{H^1(\Omega^{(\ell)})}^2 )
		+ \frac{p^2}{h_{k\ell}}\| u^{(\ell, k)} - u^{(\ell)} \|_{L_2(\Gamma^{(k,\ell)})}^2.
	\end{aligned}
	\end{equation}
	Now, we estimate $| u^{(\ell)}(\textbf x^{(k,\ell,i)}) - u^{(\ell,k)}(\mathbf{x}^{(k,\ell,i)}) |^2$ from above.
	To do so, we set $\widehat{u}^{(\ell)} = u^{(\ell)} \circ G_\ell$ and 
	$\widehat{u}^{(\ell,k)} = u^{(\ell,k)} \circ G_k$. 
	With unitary transformations, i.e., rotations and reflections, we can transform the patches such that the pre-image of $\Gamma^{(k,\ell)}$ is $(0,1) \times \{0\}$ and that the pre-image of $\mathbf{x}^{(k,\ell, i)}$ is $0$. 	

	Let $\widetilde{u}^{(\ell)}:=\widehat{u}^{(\ell)}(\cdot, 0)$
	and
	$\widetilde{u}^{(\ell,k)}:=\widehat{u}^{(\ell,k)}(\cdot, 0)$.
	Let $\eta$ be the largest value such that $\widetilde u^{(\ell)}$ and
	$\widetilde u^{(\ell, k)}$ are polynomial on $(0, \eta)$. Using the quasi-uniformity
	assumption, cf. Assumption \ref{ass:quasiuniform}, we
	obtain~$\eta \eqsim \widehat{h}_{k\ell}$. Arguments that are analogous
	to those used in the proof of Lemma~\ref{lem:max estimate},
	together with the inverse
	inequality \cite[Theorem 4.76, eq. (4.6.5)]{Schwab} yield 
	\[
		\begin{aligned}
			| \widehat{u}^{(\ell)}(0) - \widehat{u}^{(\ell,k)}(0) |^2
			&\lesssim
			\frac{1}{\eta} \| \widetilde{u}^{(\ell)} - \widetilde{u}^{(\ell, k)} \|_{L_2(0,\eta)}^2  +
			2 \| \widetilde{u}^{(\ell)} - \widetilde{u}^{(\ell, k)} \|_{L_2(0,\eta)}
			| \widetilde{u}^{(\ell)} - \widetilde{u}^{(\ell, k)} |_{H^1(0,\eta)}\\
			&\lesssim
			\frac{1}{\eta} \| \widetilde{u}^{(\ell)} - \widetilde{u}^{(\ell, k)} \|_{L_2(0,\eta)}^2 +
			\frac{p^2}{\eta} \| \widetilde{u}^{(\ell)} - \widetilde{u}^{(\ell, k)} \|_{L_2(0,\eta)}^2.
		\end{aligned}
		\]
		Since the $L_2$-norm on $(0,1)$ is larger than the $L_2$-norm on $(0,\eta)$ and $\eta\eqsim \widehat{h}_{k\ell}$ we arrive at the estimate 	
		\[
		\begin{aligned}
			| \widehat{u}^{(\ell)}(0) - \widehat{u}^{(\ell,k)}(0) |^2 
			&\lesssim
			\frac{p^2}{\widehat{h}_{k\ell}} \| \widetilde{u}^{(\ell)} - \widetilde{u}^{(\ell, k)} \|_{L_2(0,1)}^2 
			=
			\frac{p^2}{\widehat{h}_{k\ell}} \| \widehat{u}^{(\ell)} - \widehat{u}^{(\ell, k)} \|_{L_2(\widehat{\Gamma}^{(k,\ell)})}^2. 
		\end{aligned}
		\]	
		An application of \cite[Lemma 4.13]{SchneckenleitnerTakacs:2019} finally yields
	\begin{equation}\label{eq:proof:lem:corner:estimate:3}
	\begin{aligned}
		| u^{(\ell)}(\textbf x^{(k,\ell,i)}) - u^{(\ell,k)}(\mathbf{x}^{(k,\ell,i)}) |^2 
		\lesssim
		\frac{p^2}{h_{k\ell}} \| u^{(\ell)} - u^{(\ell, k)} \|_{L_2(\Gamma^{(k,\ell)})}^2.
	\end{aligned}
	\end{equation}
	The combination of \eqref{eq:proof:lem:corner:estimate:1},
	\eqref{eq:proof:lem:corner:estimate:2}
	and
	\eqref{eq:proof:lem:corner:estimate:3}
	finishes the proof.
\end{proof}

Before we give a proof of the main theorem, we estimate the sum of $w_e^{(k)}$-seminorms over all patches.

\begin{lemma}\label{lem:4:10}
	Let $u$ and $w$ be as in Lemma~\ref{lem:bbt}. Then, we have
	\[
	\begin{aligned}
		&\sum_{k=1}^K \sum_{\ell\in \mathcal{N}_\Gamma(k)} \left(
		|w^{(k)}|_{H^{1/2}(\Gamma^{(k,\ell)})}^2
		+
		|w^{(k)}|_{L^0_\infty(\Gamma^{(k,\ell)})}^2		
		\right)
		\\
		&\lesssim
		\sum_{k=1}^K\sum_{\ell\in \mathcal{N}_\Gamma(k)}
		\left( |u^{(k)}|_{H^{1/2}(\Gamma^{(k,\ell)})}^2 
		+|u^{(k)}|_{L^0_\infty(\Gamma^{(k,\ell)})}^2 
		+ \frac{p^2}{h_{k\ell}}\| u^{(k,\ell)} - u^{(k)} \|_{L_2(\Gamma^{(k,\ell)})}^2 \right)
		\\&\qquad
		+ \Lambda 
		\sum_{k=1}^K |\mathcal{H}_h^{(k)} u^{(k)}|_{H^1(\Omega^{(k)})}^2,
	\end{aligned}
	\]
	where $\Lambda:=1+\log p + \max_{k=1,\ldots,K} \log \frac{H_{k}}{h_{k}}$.
\end{lemma}
\begin{proof}
Lemma~\ref{lem:bbt} yields
\[
\begin{aligned}
	&\sum_{k=1}^K \sum_{\ell\in \mathcal{N}_\Gamma(k)}
		\left(
			|w^{(k)}|_{H^{1/2}(\Gamma^{(k,\ell)})}^2
			+ |w^{(k)}|^2_{L^0_{\infty}(\Gamma^{(k,\ell)})}
		\right)
		\\
	&  \; \lesssim
	\sum_{k=1}^K\sum_{\ell\in \mathcal{N}_\Gamma(k)}
	\left(
	  |u^{(k)}|_{H^{1/2}(\Gamma^{(k,\ell)})}^2
	  +
	  |u^{(\ell,k)}|_{H^{1/2}(\Gamma^{(k,\ell)})}^2
	  +
	  |u^{(k)}|_{L^0_\infty(\Gamma^{(k,\ell)})}^2
	  +
	  |u^{(\ell,k)}|_{L^0_\infty(\Gamma^{(k,\ell)})}^2
	  \right)
	\\ & \;\qquad
	+  \sum_{k=1}^K\sum_{\ell\in \mathcal{N}_\Gamma(k)} \sum_{i=1}^2 \Delta^{(k,\ell,i)} .
\end{aligned}
\]	  
  The Lemmas~\ref{lem:hhalf:on:artif:edge} and \ref{lem:linf:on:artif:edge}
  and the observation $\ell\in\mathcal{N}_{\Gamma}(k) \Leftrightarrow k\in\mathcal{N}_{\Gamma}(\ell)$ yield further
\[
\begin{aligned}
	&\sum_{k=1}^K \sum_{\ell\in \mathcal{N}_\Gamma(k)}
		\left(
			|w^{(k)}|_{H^{1/2}(\Gamma^{(k,\ell)})}^2
			+ |w^{(k)}|^2_{L^0_{\infty}(\Gamma^{(k,\ell)})}
		\right)
		\\
	&  \qquad \lesssim
	\sum_{k=1}^K\sum_{\ell\in \mathcal{N}_\Gamma(k)}
	\left(
	  |u^{(k)}|_{H^{1/2}(\Gamma^{(k,\ell)})}^2
	  +
	  |u^{(k)}|_{L^0_\infty(\Gamma^{(k,\ell)})}^2
	  +
	   \frac{p^2}{h_{k\ell}} \| u^{(\ell,k)} - u^{(\ell)} \|_{L_2(\Gamma^{(k,\ell)})}^2
	   \right)
	\\ & \qquad\qquad
	+  \sum_{k=1}^K\sum_{\ell\in \mathcal{N}_\Gamma(k)} \sum_{i=1}^2 \Delta^{(k,\ell,i)} .
\end{aligned}
\]	  
This finishes the proof for Alg.~A and C, since $\Delta^{(k,\ell,i)}=0$. For
Alg.~B, the desired result is a consequence of 
Lemma~\ref{lem:corner estimate} and $|\mathcal{N}_\Gamma(k)|\le 4$.
\end{proof}

\begin{lemma}\label{lem:second}
	Let $u$ and $w$ be as in Lemma~\ref{lem:bbt}. Then, we have
	\begin{align*}
	&\sum_{k=1}^{K}\sum_{\ell \in \mathcal N_\Gamma(k)} \frac{\delta p^2}{h_{k\ell}} 
	\| w^{(k)} - w^{(k,\ell)} \|_{L_2(\Gamma^{(k,\ell)})}^{2} \\
	&\qquad  \lesssim
	\Sigma 
	\sum_{k=1}^{K}\sum_{\ell \in \mathcal N_\Gamma(k)}\left(
	\frac{\delta p^2}{h_{k\ell}}\| u^{(k)} - u^{(k,\ell)} \|_{L_2(\Gamma^{(k,\ell)})}^{2}
	+ \Lambda |\mathcal{H}_h^{(k)} u^{(k)}|_{H^1(\Omega^{(k)})}^2 
	 \right),
	\end{align*}
	where 
	\begin{equation}\label{eq:sigmadef}
		\Sigma
		:=
		\begin{cases}
			1 & \mbox{for Alg.~A and C}\\
			\delta p \max_{k=1,\ldots,K}\max_{\ell\in\mathcal{N}_{\Gamma}(k)} \frac{h_{k}}{h_{\ell}} & \mbox{for Alg.~B}
		\end{cases}.
	\end{equation}
\end{lemma}
\begin{proof}
	We start with Alg.~A and C. Lemma~\ref{lem:bbt} 
	and the observation that $\ell\in\mathcal{N}_\Gamma(k)\Leftrightarrow
	k\in \mathcal{N}_\Gamma(\ell)$ yield immediately
	\begin{align*}
	&\sum_{k=1}^{K}\sum_{\ell \in \mathcal N_\Gamma(k)} \frac{\delta p^2}{h_{k\ell}} 
	\| w^{(k)} - w^{(k,\ell)} \|_{L_2(\Gamma^{(k,\ell)})}^{2} \\
	&\qquad\lesssim 
	\sum_{k=1}^{K}\sum_{\ell \in \mathcal N_\Gamma(k)}\frac{\delta p^2}{h_{k\ell}}\left(
	\| u^{(k)} - u^{(k,\ell)} \|_{L_2(\Gamma^{(k,\ell)})}^{2}
	+ \sum_{i=1}^2  \frac{h_{k}}{p} \Delta^{(k,\ell,i)}
	 \right).
	\end{align*}
	Since $\Delta^{(k,\ell,i)}=0$ for Alg.~A and C, this finishes the proof
	for this case. Now, consider the case of Alg.~B. Lemma~\ref{lem:corner estimate}
	and the observation that $\ell\in\mathcal{N}_\Gamma(k)\Leftrightarrow
	k\in \mathcal{N}_\Gamma(\ell)$ yield
	\begin{align*}
	&\sum_{k=1}^{K}\sum_{\ell \in \mathcal N_\Gamma(k)} \frac{\delta p^2}{h_{k\ell}} 
	\| w^{(k)} - w^{(k,\ell)} \|_{L_2(\Gamma^{(k,\ell)})}^{2} \\
	&\quad\lesssim 
	\sum_{k=1}^{K}\sum_{\ell \in \mathcal N_\Gamma(k)}\frac{\delta p^2}{h_{k\ell}}\Big(
	\| u^{(k)} - u^{(k,\ell)} \|_{L_2(\Gamma^{(k,\ell)})}^{2}
	 \\ & \qquad\qquad + 
	\Lambda \frac{h_{k} + h_{\ell}}{p} |\mathcal{H}_h^{(k)} u^{(k)}|_{H^1(\Omega^{(k)})}^2 
		+ \frac{p h_{k}}{h_{k\ell}} 
					\| u^{(\ell)} - u^{(\ell, k)} \|_{L_2(\Gamma^{(k,\ell)})}^2
	 \Big).
	\end{align*}
	Using $h_{k\ell}=\min\{h_{k},h_{\ell}\}$
	and $\ell\in\mathcal{N}_\Gamma(k)\Leftrightarrow
	k\in \mathcal{N}_\Gamma(\ell)$,
	we immediately obtain the desired bound for Alg.~B.
\end{proof}

Now we are able to prove the bound for the condition number of the preconditioned dG IETI method as stated in Theorem~\ref{thrm:fin}.
\begin{proof}[Proof of Theorem~\ref{thrm:fin}]
	The idea of the proof is to use \cite[Theorem~22]{MandelDohrmannTezaur:2005a},
	which states that 
	\begin{equation}\label{eq:MandelDohrmannTezaur}
			\kappa(M_{\mathrm{sD}} \, F) \le \sup_{u \in  \widetilde W }
			\frac{ \| B_D^\top B \underline u \|_S^2 }{ \| \underline u \|_S^2 },
	\end{equation}
	where $\underline u$ is the coefficient vector associated
	to the function $u=(u_e^{(1)},\cdots,u_e^{(K)})
	=((u^{(1)},(u^{(1,\ell)})_{\ell\in\mathcal N_\Gamma(k)}),\cdots)$.
	So, let $u$ be arbitrary but fixed
	and let the function $w=(w_e^{(1)},\cdots,w_e^{(K)})
	=((w^{(1)},(w^{(1,\ell)})_{\ell\in\mathcal N_\Gamma(k)}),\cdots)$
	with coefficient vector $\underline w$ be such that
	$
			\underline w = B_D^\top B \underline u
	$.
	Moreover, let $v_e^{(k)} \in V_e^{(k)}$ be an arbitrary extension of $w_e^{(k)}$.  Lemma~\ref{lem:patchwise equ} yields 
	\[
	\begin{aligned}
			&\| B_D^\top B \underline u \|_S^2
			   = \|  \underline w \|_S^2
			   = \sum_{k=1}^K \|  \mathcal H_A^{(k)} w_e^{(k)} \|_{a_e^{(k)}}^2 \\
			   &\quad= \sum_{k=1}^{K} \inf_{w_0^{(k)} \in V^{(k)}_0}\| v_e^{(k)} - (w_0^{(k)},0^{|\mathcal N_\Gamma(k)|})\|_{a_e^{(k)}}^2
			   \eqsim
			   \sum_{k=1}^{K}\inf_{w_0^{(k)} \in V^{(k)}_0}\| v_e^{(k)} - (w_0^{(k)},0^{|\mathcal N_\Gamma(k)|})\|_{d_e^{(k)}}^2 \\
			   &\quad  = 
			   \sum_{k=1}^{K}\inf_{w_0^{(k)} \in V^{(k)}_0}
			   \left( | v^{(k)} - w_0^{(k)} |_{H^1(\Omega^{(k)})}^2  
			   +
			   \frac{\delta p^2}{h_{k\ell}} \sum_{\ell \in \mathcal N_\Gamma(k)}  
			   \| w^{(k)} -w^{(k,\ell)}  \|_{L_2(\Gamma^{(k,\ell)})}^{2}
			   \right).
	\end{aligned}
	\]
	Note that the second sum does not depend on $w_0^{(k)}$. Thus, the infimum
	refers only to the $H^1$-seminorm, which means that that seminorm coincides
	with the seminorm of the (standard) discrete harmonic extension.
	Hence, we have
	\begin{equation}\label{eq:thetwosums}
	\begin{aligned}
			   \| B_D^\top B \underline u \|_S^2
			   \eqsim \sum_{k=1}^{K} | \mathcal{H}_h^{(k)} w^{(k)} |_{H^{1}(\Omega^{(k)})}^2
			   + \sum_{k=1}^{K} \sum_{\ell\in\mathcal{N}_\Gamma(k)}
			   \frac{\delta p^2}{h_{kl}}
			   \| w^{(k)} -w^{(k,\ell)} \|_{L_2(\Gamma^{(k,\ell)})}^2.
	\end{aligned}
	\end{equation}
	First, we estimate the first sum in~\eqref{eq:thetwosums}.
	\cite[Theorem 4.2]{SchneckenleitnerTakacs:2019} yields
	\[
		\sum_{k=1}^{K} |\mathcal{H}_h^{(k)} w^{(k)} |_{H^1(\Omega^{(k)})}^2
		\lesssim
		p\sum_{k=1}^{K} | w^{(k)} |_{H^{1/2}(\partial \Omega^{(k)})}^2.
	\]
	Using \cite[Lemma 4.15]{SchneckenleitnerTakacs:2019}, we get
	\[
	\begin{aligned}
		\sum_{k=1}^{K} |\mathcal{H}_h^{(k)} w^{(k)} |_{H^1(\Omega^{(k)})}^2 
	&	\lesssim
		p\sum_{k=1}^{K} \sum_{\ell \in \mathcal N_\Gamma(k)} \left( | w^{(k)} |_{H^{1/2}(\Gamma^{(k,\ell)})}^2 + 
		\Lambda
		|w^{(k)}|_{L_\infty^0(\Gamma^{(k,\ell)})}^2 
		\right).
	\end{aligned}	
	\]
	Using $\Lambda\ge1$ and Lemma~\ref{lem:4:10}, we obtain further
	\[
	\begin{aligned}
		&\sum_{k=1}^{K} |\mathcal{H}_h^{(k)} w^{(k)} |_{H^1(\Omega^{(k)})}^2 
		\lesssim
		p  \Lambda 
		\sum_{k=1}^K\sum_{\ell\in \mathcal{N}_\Gamma(k)}
		\left( |u^{(k)}|_{H^{1/2}(\Gamma^{(k,\ell)})}^2 
		+|u^{(k)}|_{L_\infty^0(\Gamma^{(k,\ell)})}^2
		\right)\\
		&\qquad+
		p  \Lambda 
		\sum_{k=1}^K\sum_{\ell\in \mathcal{N}_\Gamma(k)}
		 \frac{p^2}{h_{k\ell}}\| u^{(k,\ell)} - u^{(k)} \|_{L_2(\Gamma^{(k,\ell)})}^2
		+ p \Lambda^2
		\sum_{k=1}^K |\mathcal{H}_h^{(k)} u^{(k)}|_{H^1(\Omega^{(k)})}^2.
	\end{aligned}
	\]
	Using \cite[Lemma~4.15 and Theorem~4.2]{SchneckenleitnerTakacs:2019}
	and $|\mathcal{N}_\Gamma(k)|\le4$, we further estimate
	\[
	\begin{aligned}
	\sum_{k=1}^{K} |\mathcal{H}_h^{(k)} w^{(k)} |_{H^1(\Omega^{(k)})}^2 
	& \lesssim p\Lambda^2
	\sum_{k=1}^K |\mathcal{H}_h^{(k)}u^{(k)}|_{H^{1}(\Omega^{(k)})}^2 
	+ p \Lambda  \sum_{k=1}^K \sum_{\ell\in \mathcal{N}_\Gamma(k)}|u^{(k)}|_{L_\infty^0(\Gamma^{(k,\ell)})}^2
	\\ & \hspace{-2cm} + p \Lambda
	\sum_{k=1}^K \sum_{\ell\in \mathcal{N}_\Gamma(k)}
	 \frac{p^2}{h_{k\ell}}\| u^{(k,\ell)} - u^{(k)} \|_{L_2(\Gamma^{(k,\ell)})}^2 
	.
	\end{aligned}
	\]
	Using \cite[Lemma~4.14]{SchneckenleitnerTakacs:2019}, we get
	further
	\[
	\begin{aligned}
	& \sum_{k=1}^{K} |\mathcal{H}_h^{(k)} w^{(k)} |_{H^1(\Omega^{(k)})}^2
	\lesssim
	p\Lambda^2
	\sum_{k=1}^K |\mathcal{H}_h^{(k)}u^{(k)}|_{H^{1}(\Omega^{(k)})}^2  \\
	& \quad 
	+ p \Lambda^2  \sum_{k=1}^K \sum_{\ell\in \mathcal{N}_\Gamma(k)}
	\left(
		|\mathcal{H}_h^{(k)}u^{(k)}|_{H^{1}(\Omega^{(k)})}^2 + (H_{k})^{-2}\inf_{c\in \mathbb R} 
		\|\mathcal{H}_h^{(k)}u^{(k)}-c\|_{L_2(\Omega^{(k)})}^2 
	\right)\\
	&\quad + p \Lambda
	\sum_{k=1}^K \sum_{\ell\in \mathcal{N}_\Gamma(k)}
	 \frac{p^2}{h_{k\ell}}\| u^{(k,\ell)} - u^{(k)} \|_{L_2(\Gamma^{(k,\ell)})}^2 
	.
	\end{aligned}
	\]
	$|\mathcal{N}_\Gamma(k)|\lesssim 1$ and the Poincar{\'e} inequality
	yield the estimate 
	\begin{equation}\label{eq:AlgA first}
	\begin{aligned}
	& \sum_{k=1}^{K} |\mathcal{H}_h^{(k)} w^{(k)} |_{H^1(\Omega^{(k)})}^2 \\
	& \quad \lesssim
	p\Lambda^2
	\sum_{k=1}^K
	\left(
	|\mathcal{H}_h^{(k)}u^{(k)}|_{H^{1}(\Omega^{(k)})}^2 
	+  \sum_{\ell\in \mathcal{N}_\Gamma(k)}
	 \frac{p^2}{h_{k\ell}}\| u^{(k,\ell)} - u^{(k)} \|_{L_2(\Gamma^{(k,\ell)})}^2
	\right).
	\end{aligned}
	\end{equation}
	To estimate the second sum in~\eqref{eq:thetwosums},
	we use Lemma~\ref{lem:second} and obtain
	\begin{equation}\label{eq:AlgA second}
	\begin{aligned}
	&\sum_{k=1}^{K}\sum_{\ell \in \mathcal N_\Gamma(k)} \frac{\delta p^2}{h_{k\ell}} 
	\| w^{(k)} - w^{(k,\ell)} \|_{L_2(\Gamma^{(k,\ell)})}^{2} \\
	&\qquad  \lesssim
	\Sigma 
	\sum_{k=1}^{K}\sum_{\ell \in \mathcal N_\Gamma(k)}\left(
	\frac{\delta p^2}{h_{k\ell}}\| u^{(k)} - u^{(k,\ell)} \|_{L_2(\Gamma^{(k,\ell)})}^{2}
	+ \Lambda |\mathcal{H}_h^{(k)} u^{(k)}|_{H^1(\Omega^{(k)})}^2 
	 \right),
	\end{aligned}
	\end{equation}
	where $\Sigma$ is as in~\eqref{eq:sigmadef}.
	The combination of $\delta\gtrsim 1$, $\Lambda\ge 1$, $\Sigma \ge 1$,\eqref{eq:thetwosums}, \eqref{eq:AlgA first} and
	\eqref{eq:AlgA second} yields
	\[
	\begin{aligned}
		\| B_D^\top B \underline u \|_S^2 &\lesssim
		(p \Lambda^2+\Sigma \Lambda) \sum_{k=1}^{K} \left(
			| \mathcal{H}_h^{(k)} u^{(k)} |_{H^1(\Omega^{(k)})}^2 + 
			\sum_{\ell \in \mathcal N_\Gamma(k)}\frac{\delta p^2}{h_{k\ell}}
			\| u^{(k)} - u^{(k,\ell)} \|_{L_2(\Gamma^{(k,\ell)})}^{2} 
		\right) \\
		&= (p \Lambda^2+\Sigma \Lambda) \| u \|_{d}^2 
		\eqsim (p \Lambda^2+\Sigma \Lambda) \| \underline{u} \|_{S}^2 
		\lesssim \Sigma \Lambda^2 \| \underline{u} \|_{S}^2. 
	\end{aligned}
	\]
	The combination of this estimate and~\eqref{eq:MandelDohrmannTezaur}
	finishes the proof.
\end{proof}

\section{Numerical results}
\label{sec:5}
In this section, we present the results of our numerical experiments that
illustrate the presented convergence theory.
We consider the Poisson problem 
\[
\begin{aligned}
		- \Delta u(x,y) & = 2\pi^2 \sin(\pi x)\sin(\pi y) &&\qquad \mbox{for}\quad (x,y)\in\Omega \\
				  u & = 0 &&\qquad \mbox{on}\quad \partial\Omega,
\end{aligned}
\]
where the computational domain $\Omega$ is one of the domains depicted in
Figure~\ref{fig:computational domains}.

\begin{figure}[h]
	\centering
	\begin{subfigure}{.27\textwidth}
		\centering
		\includegraphics[width = 4.2cm, height = 4.2cm]{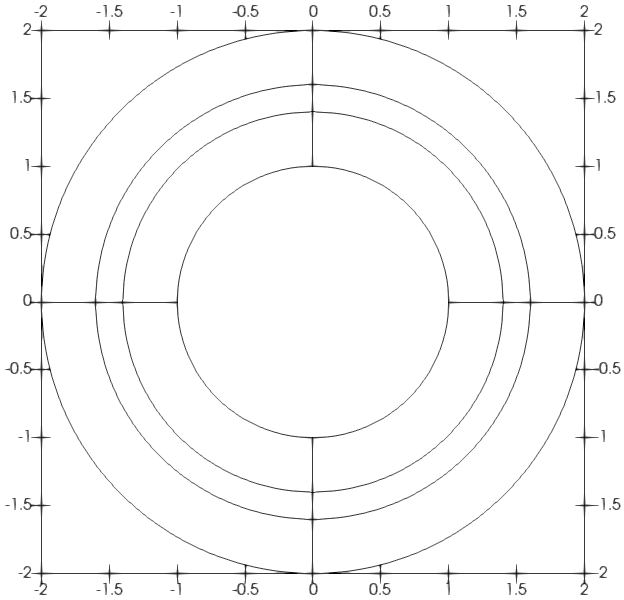}
		{\tiny \caption{Circular ring}\label{subfig:Rings}}
	\end{subfigure} 
	\hspace{0.5cm}
	\begin{subfigure}{.27\textwidth}
		\centering
		\includegraphics[height = 4.2cm]{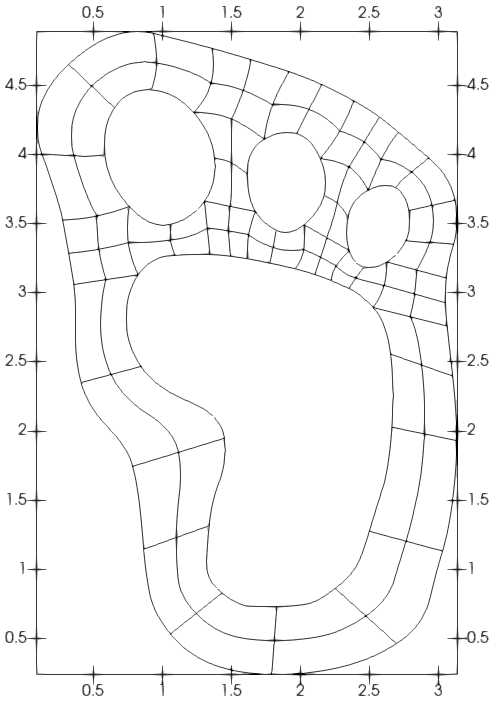}
		{\tiny \caption{Yeti-footprint}
		\label{subfig:Yeti}}
	\end{subfigure} 
	\caption{Computational domains and the decomposition into patches}
	\label{fig:computational domains}
\end{figure}

The first domain (Figure~\ref{subfig:Rings})
is a circular ring consisting of 12 patches. Each of them is the image
of a NURBS mapping of degree 2. The second domain (Figure~\ref{subfig:Yeti})
is the Yeti-footprint which is composed of 84 patches. In this case, all of
the patches are parameterized using B-spline functions, again of degree 2.

\begin{table}[t]
	\newcolumntype{L}[1]{>{\raggedleft\arraybackslash\hspace{-1em}}m{#1}}
	\centering
	\renewcommand{\arraystretch}{1.25}
	\begin{tabular}{l|L{1em}L{1.8em}|L{1em}L{1.8em}|L{1em}L{1.8em}|L{1em}L{1.8em}|L{1em}L{1.8em}|L{1em}L{1.8em}|L{1em}L{1.8em}}
		\toprule
	      \multicolumn{1}{l}{r$\;\;\diagdown\;\;$p\hspace{-1.8em}\;}
		& \multicolumn{2}{c|}{2}
		& \multicolumn{2}{c|}{3}
		& \multicolumn{2}{c|}{4}
		& \multicolumn{2}{c|}{5}
		& \multicolumn{2}{c|}{6}
		& \multicolumn{2}{c|}{7}
		& \multicolumn{2}{c}{8} \\
		& it & $\kappa$
		& it & $\kappa$
		& it & $\kappa$
		& it & $\kappa$
		& it & $\kappa$
		& it & $\kappa$
		& it & $\kappa$ \\
		\midrule
		$2$  & $10$ & $5.65$ & $12$ & $6.23$ & $12$ & $6.53$ & $13$ & $7.06$ & $13$ & $7.28$ & $14$ & $7.77$ & $14$ & $7.93$ \\
		$3$  & $12$ & $6.29$ & $12$ & $6.93$ & $13$ & $7.44$ & $13$ & $7.90$ & $14$ & $8.33$ & $14$ & $8.67$ & $15$ & $9.03$ \\
		$4$  & $13$ & $7.04$ & $13$ & $7.88$ & $14$ & $8.53$ & $14$ & $9.09$ & $15$ & $9.59$ & $15$ & $10.00$ & $15$ & $10.37$ \\
		$5$  & $13$ & $8.00$ & $14$ & $9.06$ & $15$ & $9.82$ & $15$ & $10.47$ & $16$ & $11.03$ & $16$ & $11.49$ & $16$ & $11.91$ \\
		$6$  & $15$ & $9.19$ & $16$ & $10.42$ & $16$ & $11.29$ & $17$ & $12.02$ & $16$ & $12.64$ & $18$ & $13.16$ & $17$ & $13.63$ \\
		$7$  & $16$ & $10.58$ & $17$ & $11.97$ & $18$ & $12.94$ & $18$ & $13.74$ & $18$ & $14.43$ & $18$ & $15.01$ & $17$ & $15.52$ \\
		$8$  & $17$ & $12.14$ & $18$ & $13.69$ & $18$ & $14.76$ & $19$ & $15.65$ & $19$ & $16.40$ & $19$ & $17.02$ & $19$ & $17.58$ \\
		\bottomrule
	\end{tabular}
	\captionof{table}{Iteration counts (it) and condition numbers ($\kappa$); Alg.~A; triple ring
	\label{tab:TripleRing_vertex}}
\end{table}
\begin{table}[t]
	\newcolumntype{L}[1]{>{\raggedleft\arraybackslash\hspace{-1em}}m{#1}}
	\centering
	\renewcommand{\arraystretch}{1.25}
	\resizebox{\columnwidth}{!}{
	\begin{tabular}{l|L{1em}L{1.8em}|L{1em}L{1.8em}|L{1em}L{1.8em}|L{1em}L{1.8em}|L{1em}L{1.8em}|L{1em}L{1.8em}|L{1em}L{1.8em}}
		\toprule
	      \multicolumn{1}{l}{r$\;\;\diagdown\;\;$p\hspace{-1.8em}\;}
		& \multicolumn{2}{c|}{2}
		& \multicolumn{2}{c|}{3}
		& \multicolumn{2}{c|}{4}
		& \multicolumn{2}{c|}{5}
		& \multicolumn{2}{c|}{6}
		& \multicolumn{2}{c|}{7}
		& \multicolumn{2}{c}{8} \\
		& it & $\kappa$
		& it & $\kappa$
		& it & $\kappa$
		& it & $\kappa$
		& it & $\kappa$
		& it & $\kappa$
		& it & $\kappa$ \\
		\midrule
		$2$  & $36$ & $275$ & $43$ & $ 348$ & $50$ & $414$ & $49$ & $472$ & $52$ & $547$ & $52$ & $618$ & $56$ & $727$ \\
		$3$  & $38$ & $294$ & $44$ & $361$ & $48$ & $433$ & $50$ & $521$ & $52$ & $571$ & $54$ & $678$ & $55$ & $739$ \\
		$4$  & $44$ & $311$ & $49$ & $396$ & $49$ & $452$ & $51$ & $544$ & $52$ & $605$ & $55$ & $700$ & $55$ & $764$ \\
		$5$  & $48$ & $342$ & $47$ & $401$ & $50$ & $477$ & $51$ & $567$ & $53$ & $649$ & $55$ & $723$ & $57$ & $830$ \\
		$6$  & $50$ & $357$ & $50$ & $425$ & $52$ & $510$ & $54$ & $585$ & $55$ & $681$ & $56$ & $759$ & $58$ & $866$ \\
		$7$  & $49$ & $379$ & $52$ & $461$ & $55$ & $554$ & $57$ & $628$ & $58$ & $708$ & $60$ & $811$ & $62$ & $903$ \\
		$8$  & $53$ & $404$ & $53$ & $474$ & $56$ & $577$ & $59$ & $677$ & $59$ & $753$ & $64$ & $878$ & $65$ & $978$ \\
		\bottomrule
	\end{tabular}}
	\captionof{table}{Iteration counts (it) and condition numbers ($\kappa$); Alg.~B; triple ring
		\label{tab:TripleRing_edge}}
\end{table}
\begin{table}[t]
	\newcolumntype{L}[1]{>{\raggedleft\arraybackslash\hspace{-1em}}m{#1}}
	\centering
	\renewcommand{\arraystretch}{1.25}
	\begin{tabular}{l|L{1em}L{1.8em}|L{1em}L{1.8em}|L{1em}L{1.8em}|L{1em}L{1.8em}|L{1em}L{1.8em}|L{1em}L{1.8em}|L{1em}L{1.8em}}
		\toprule
	      \multicolumn{1}{l}{r$\;\;\diagdown\;\;$p\hspace{-1.8em}\;}
		& \multicolumn{2}{c|}{2}
		& \multicolumn{2}{c|}{3}
		& \multicolumn{2}{c|}{4}
		& \multicolumn{2}{c|}{5}
		& \multicolumn{2}{c|}{6}
		& \multicolumn{2}{c|}{7}
		& \multicolumn{2}{c}{8} \\
		& it & $\kappa$
		& it & $\kappa$
		& it & $\kappa$
		& it & $\kappa$
		& it & $\kappa$
		& it & $\kappa$
		& it & $\kappa$ \\
		\midrule
		$2$  & $7$ & $1.60$ & $10$ & $2.25$ & $10$ & $2.45$ & $11$ & $2.55$ & $11$ & $2.80$ & $12$ & $2.89$ & $12$ & $3.13$ \\
		$3$  & $10$ & $2.28$ & $10$ & $2.56$ & $11$ & $2.81$ & $11$ & $3.04$ & $11$ & $3.24$ & $12$ & $3.43$ & $12$ & $3.58$ \\
		$4$  & $11$ & $2.62$ & $11$ & $3.03$ & $12$ & $3.35$ & $11$ & $3.63$ & $12$ & $3.88$ & $13$ & $4.08$ & $13$ & $4.27$ \\
		$5$  & $11$ & $3.09$ & $12$ & $3.62$ & $12$ & $4.00$ & $13$ & $4.32$ & $13$ & $4.59$ & $14$ & $4.81$ & $14$ & $5.01$ \\
		$6$  & $12$ & $3.69$ & $13$ & $4.30$ & $14$ & $4.72$ & $14$ & $5.07$ & $14$ & $5.36$ & $15$ & $5.61$ & $15$ & $5.82$ \\
		$7$  & $13$ & $4.37$ & $14$ & $5.04$ & $15$ & $5.50$ & $15$ & $5.88$ & $16$ & $6.20$ & $16$ & $6.46$ & $16$ & $6.69$ \\
		$8$  & $14$ & $5.13$ & $15$ & $5.86$ & $16$ & $6.35$ & $16$ & $6.75$ & $17$ & $7.09$ & $17$ & $7.37$ & $17$ & $7.61$ \\
		\bottomrule
	\end{tabular}
	\captionof{table}{Iteration counts (it) and condition numbers ($\kappa$); Alg.~C; triple ring
		\label{tab:TripleRing_vertexEdge}}
\end{table}

\begin{table}[t]
	\newcolumntype{L}[1]{>{\raggedleft\arraybackslash\hspace{-1em}}m{#1}}
	\centering
	\renewcommand{\arraystretch}{1.25}
	\begin{tabular}{l|L{1em}L{1.8em}|L{1em}L{1.8em}|L{1em}L{1.8em}|L{1em}L{1.8em}|L{1em}L{1.8em}|L{1em}L{1.8em}|L{1em}L{1.8em}}
		\toprule
	      \multicolumn{1}{l}{r$\;\;\diagdown\;\;$p\hspace{-1.8em}\;}
		& \multicolumn{2}{c|}{2}
		& \multicolumn{2}{c|}{3}
		& \multicolumn{2}{c|}{4}
		& \multicolumn{2}{c|}{5}
		& \multicolumn{2}{c|}{6}
		& \multicolumn{2}{c|}{7}
		& \multicolumn{2}{c}{8}  \\
		& it & $\kappa$
		& it & $\kappa$
		& it & $\kappa$
		& it & $\kappa$
		& it & $\kappa$
		& it & $\kappa$
		& it & $\kappa$  \\
		\midrule
		$1$  & $10$ & $2.09$ & $12$ & $2.73$ & $14$ & $3.27$ & $15$ & $3.80$ & $16$ & $4.20$ & $18$ & $4.67$ & $19$ & $5.00$ \\
		$2$  & $12$ & $2.82$ & $14$ & $3.56$ & $15$ & $4.11$ & $16$ & $4.60$ & $17$ & $5.01$ & $19$ & $5.43$ & $20$ & $5.77$ \\
		$3$  & $14$ & $3.77$ & $16$ & $4.61$ & $17$ & $5.24$ & $18$ & $5.77$ & $19$ & $6.22$ & $20$ & $6.63$ & $21$ & $6.97$ \\
		$4$  & $16$ & $4.86$ & $17$ & $5.83$ & $19$ & $6.54$ & $20$ & $7.13$ & $21$ & $7.64$ & $22$ & $8.06$ & $23$ & $8.46$ \\
		$5$  & $19$ & $6.12$ & $20$ & $7.21$ & $21$ & $8.00$ & $22$ & $8.65$ & $23$ & $9.21$ & $24$ & $9.73$ & $24$ & $10.13$ \\
		$6$  & $21$ & $7.54$ & $22$ & $8.74$ & $23$ & $9.61$ & $24$ & $10.37$ & $25$ & $10.98$ & $25$ & $11.50$ & $26$ & $11.96$ \\
		$7$  & $22$ & $9.11$ & $24$ & $10.44$ & $25$ & $11.38$ & $26$ & $12.23$ & $27$ & $12.89$ & $27$ & $13.50$ & $27$ & $14.02$ \\
		\bottomrule
	\end{tabular}
	\captionof{table}{Iteration counts (it) and condition numbers ($\kappa$); Alg.~A; Yeti-footprint
		\label{tab:Yeti_vertex}}
\end{table}
\begin{table}[t]
	\newcolumntype{L}[1]{>{\raggedleft\arraybackslash\hspace{-1em}}m{#1}}
	\centering
	\renewcommand{\arraystretch}{1.25}
	\resizebox{\columnwidth}{!}{
	\begin{tabular}{l|L{1em}L{1.8em}|L{1em}L{1.8em}|L{1em}L{1.8em}|L{1em}L{1.8em}|L{1em}L{1.8em}|L{1em}L{1.8em}|L{1em}L{1.8em}}
		\toprule
	      \multicolumn{1}{l}{r$\;\;\diagdown\;\;$p\hspace{-1.8em}\;}
		& \multicolumn{2}{c|}{2}
		& \multicolumn{2}{c|}{3}
		& \multicolumn{2}{c|}{4}
		& \multicolumn{2}{c|}{5}
		& \multicolumn{2}{c|}{6}
		& \multicolumn{2}{c|}{7}
		& \multicolumn{2}{c}{8} \\
		& it & $\kappa$
		& it & $\kappa$
		& it & $\kappa$
		& it & $\kappa$
		& it & $\kappa$
		& it & $\kappa$
		& it & $\kappa$ \\
		\midrule
		$1$  & $42$ & $27$ & $49$ & $39$ & $57$ & $52$ & $65$ & $68$ & $72$ & $80$ & $82$ & $97$ & $90$ & $117$ \\
		$2$  & $44$ & $33$ & $50$ & $45$ & $56$ & $58$ & $61$ & $74$ & $69$ & $88$ & $75$ & $103$ & $84$ & $122$ \\
		$3$  & $49$ & $41$ & $54$ & $56$ & $59$ & $70$ & $64$ & $87$ & $70$ & $100$ & $76$ & $119$ & $81$ & $132$ \\
		$4$  & $53$ & $51$ & $60$ & $65$ & $63$ & $84$ & $69$ & $97$ & $74$ & $119$ & $80$ & $132$ & $85$ & $156$ \\
		$5$  & $59$ & $57$ & $64$ & $75$ & $70$ & $90$ & $75$ & $112$ & $80$ & $135$ & $84$ & $149$ & $90$ & $171$ \\
		$6$  & $62$ & $66$ & $68$ & $83$ & $74$ & $101$ & $79$ & $127$ & $85$ & $147$ & $89$ & $170$ & $96$ & $191$ \\
		$7$  & $65$ & $75$ & $74$ & $97$ & $79$ & $117$ & $83$ & $146$ & $88$ & $164$ & $94$ & $188$ & $100$ & $215$ \\
		\bottomrule
	\end{tabular}}
	\captionof{table}{Iteration counts (it) and condition numbers ($\kappa$); Alg.~B; Yeti-footprint
		\label{tab:Yeti_edge}}
\end{table}

\begin{table}[t]
	\newcolumntype{L}[1]{>{\raggedleft\arraybackslash\hspace{-1em}}m{#1}}
	\centering
	\renewcommand{\arraystretch}{1.25}
	\begin{tabular}{l|L{1em}L{1.8em}|L{1em}L{1.8em}|L{1em}L{1.8em}|L{1em}L{1.8em}|L{1em}L{1.8em}|L{1em}L{1.8em}|L{1em}L{1.8em}}
		\toprule
	      \multicolumn{1}{l}{r$\;\;\diagdown\;\;$p\hspace{-1.8em}\;}
		& \multicolumn{2}{c|}{2}
		& \multicolumn{2}{c|}{3}
		& \multicolumn{2}{c|}{4}
		& \multicolumn{2}{c|}{5}
		& \multicolumn{2}{c|}{6}
		& \multicolumn{2}{c|}{7}
		& \multicolumn{2}{c}{8} \\
		& it & $\kappa$
		& it & $\kappa$
		& it & $\kappa$
		& it & $\kappa$
		& it & $\kappa$
		& it & $\kappa$
		& it & $\kappa$ \\
		\midrule
		$1$  & $6$ & $1.16$ & $7$ & $1.27$ & $9$ & $1.43$ & $10$ & $1.57$ & $11$ & $1.71$ & $12$ & $1.83$ & $12$ & $1.95$ \\
		$2$  & $7$ & $1.32$ & $9$ & $1.49$ & $10$ & $1.67$ & $11$ & $1.81$ & $12$ & $1.95$ & $13$ & $2.09$ & $13$ & $2.19$ \\
		$3$  & $9$ & $1.57$ & $10$ & $1.81$ & $11$ & $2.00$ & $12$ & $2.17$ & $12$ & $2.28$ & $14$ & $2.46$ & $15$ & $2.58$ \\
		$4$  & $11$ & $1.89$ & $12$ & $2.18$ & $13$ & $2.41$ & $13$ & $2.59$ & $14$ & $2.76$ & $15$ & $2.91$ & $16$ & $3.05$ \\
		$5$  & $12$ & $2.26$ & $13$ & $2.61$ & $14$ & $2.87$ & $15$ & $3.06$ & $16$ & $3.27$ & $16$ & $3.42$ & $17$ & $3.57$ \\
		$6$  & $14$ & $2.70$ & $15$ & $3.09$ & $16$ & $3.38$ & $16$ & $3.61$ & $17$ & $3.82$ & $18$ & $4.00$ & $19$ & $4.17$ \\
		$7$  & $15$ & $3.19$ & $16$ & $3.63$ & $17$ & $3.95$ & $18$ & $4.21$ & $18$ & $4.43$ & $19$ & $4.63$ & $20$ & $4.81$ \\
		\bottomrule
	\end{tabular}
	\captionof{table}{Iteration counts (it) and condition numbers ($\kappa$); Alg.~C; Yeti-footprint
		\label{tab:Yeti_vertexEdge}}
\end{table}

\begin{table}[t]
	\newcolumntype{L}[1]{>{\raggedleft\arraybackslash\hspace{-1em}}m{#1}}
	\centering
	\renewcommand{\arraystretch}{1.25}
	\begin{tabular}{l|L{1em}L{1.8em}|L{1em}L{1.8em}|L{1em}L{1.8em}|L{1em}L{1.8em}|L{1em}L{1.8em}|L{1em}L{1.8em}|L{1em}L{1.8em}}
		\toprule
		\multicolumn{1}{l}{e$\;\;\diagdown\;\;$p\hspace{-1.8em}\;}
		& \multicolumn{2}{c|}{2}
		& \multicolumn{2}{c|}{3}
		& \multicolumn{2}{c|}{4}
		& \multicolumn{2}{c|}{5}
		& \multicolumn{2}{c|}{6}
		& \multicolumn{2}{c|}{7}
		& \multicolumn{2}{c}{Alg.} \\
		& it & $\kappa$
		& it & $\kappa$
		& it & $\kappa$
		& it & $\kappa$
		& it & $\kappa$
		& it & $\kappa$
		&  & $ $ \\
		\midrule
		$1$  &  $14$ &  $8$ &  $14$ &  $9$ &  $15$ &  $9$ &  $16$ &  $10$ &  $16$ &  $10$ & $16$ &  $11$ & \rdelim\}{3}{0.4cm}[\; \normalfont A] \\
		$2$  &  $14$ &  $8$ &  $15$ &  $9$ & $15$ &  $9$ & $16$ &  $10$ & $16$ &  $10$ & $17$ &  $11$ & \\
		$3$  & $14$ &  $8$ & $15$ &  $9$ & $15$ &  $9$ & $16$ &  $10$ & $17$ &  $10$ & $18$ &  $11$ & \\
		$1$  &  $49$ &  $338$ &  $50$ &  $437$ &  $52$ &  $500$ &  $55$ &  $622$ &  $58$ &  $698$ & $61$ &  $769$ & \rdelim\}{3}{0.4cm}[\; \normalfont B] \\
		$2$  &  $52$ &  $486$ &  $55$ &  $603$ & $59$ &  $743$ & $65$ &  $893$ & $67$ &  $1041$ & $72$ &  $1196$ \\
		$3$  & $64$ &  $852$ & $71$ &  $1072$ & $75$ &  $1317$ & $77$ &  $1619$ & $80$ &  $1841$ & $89$ &  $2122$ \\
		\bottomrule
	\end{tabular}
	\captionof{table}{Iteration counts (it) and condition numbers ($\kappa$) depending on grid size disparity; triple ring
		\label{tab:ratio dependence ring}}
\end{table}

\begin{table}[t]
	\newcolumntype{L}[1]{>{\raggedleft\arraybackslash\hspace{-1em}}m{#1}}
	\centering
	\renewcommand{\arraystretch}{1.25}
	\begin{tabular}{l|L{1.4em}L{1.4em}|L{1.4em}L{1.4em}|L{1.4em}L{1.4em}|L{1.4em}L{1.4em}|L{1.4em}L{1.4em}|L{1.4em}L{1.4em}|L{1.4em}L{1.4em}}
		\toprule
		\multicolumn{1}{l}{e$\;\;\diagdown\;\;$p\hspace{-1.8em}\;}
		& \multicolumn{2}{c|}{2}
		& \multicolumn{2}{c|}{3}
		& \multicolumn{2}{c|}{4}
		& \multicolumn{2}{c|}{5}
		& \multicolumn{2}{c|}{6}
		& \multicolumn{2}{c|}{7}
		& \multicolumn{2}{c}{Alg.} \\
		& it & $\kappa$
		& it & $\kappa$
		& it & $\kappa$
		& it & $\kappa$
		& it & $\kappa$
		& it & $\kappa$
		&  & $ $ \\
		\midrule
		$1$  & $15$ & $5$ & $17$ & $5$ & $18$ & $6$ & $19$ &  $7$ & $19$ & $7$ & $21$ & $7$ & \rdelim\}{3}{0.4cm}[\; \normalfont A] \\
		$2$  &  $17$ &  $5$ &  $18$ &  $6$ & $19$ & $7$ & $19$ &  $7$ & $20$ &  $8$ & $22$ &  $8$ & \\
		$3$  & $17$ &  $6$ & $19$ &  $7$ & $19$ &  $8$ & $20$ &  $8$ & $21$ &  $9$ & $22$ &  $9$ & \\
		$1$  &  $54$ &  $52$ &  $60$ &  $69$ &  $65$ &  $83$ &  $70$ &  $99$ &  $73$ &  $120$ & $79$ &  $137$ & \rdelim\}{3}{0.4cm}[\; \normalfont B] \\
		$2$  &  $62$ &  $72$ &  $69$ &  $96$ & $76$ &  $123$ & $78$ &  $151$ & $84$ &  $177$ & $90$ &  $215$ & \\
		$3$  & $77$ &  $133$ & $87$ &  $182$ & $95$ &  $233$ & $103$ &  $278$ & $110$ &  $335$ & $115$ &  $390$ & \\
		\bottomrule
	\end{tabular}
	\captionof{table}{Iteration counts (it) and condition numbers ($\kappa$) depending on grid size disparity;  Yeti-footprint
		\label{tab:ratio dependence yeti}}
\end{table}

The numerical experiments are carried out with B-spline discretization spaces
of maximum smoothness on grids that are constructed as follows.
For the circular ring, the
coarsest grid on each patch only consists of one element, i.e., the discretization
space (on the parameter domain) for each patch consists only of
polynomial functions. For the Yeti-footprint, the coarsest grid for the 
$20$ patches in the bottom of the domain consists of two elements per patch. 
The grid is constructed by adding an edge that connects the midpoints of
the two longer sides of the patch.
The other patches of the Yeti-footprint only consist of one element.
For both domains, the finer grids are constructed by refinement. For the first refinement step, i.e., for $r=1$, we insert one knot into each knot span. The new knot is not located in the
center, but at $4/9$ times the length of the knot span if the patch index
$k$ is even and at $6/11$ times that length if $k$ is odd. The subsequent
refinement steps $r=2,3,\dots$ are done uniformly. This refinement procedure
yields discretizations that are non-matching at the interfaces.

For these discretization spaces, we set up a dG IETI discretization as introduced
in Section~\ref{sec:3}. 
To solve the system \eqref{IETIProblem}, we use a preconditioned conjugate gradient (PCG) method with the proposed scaled Dirichlet preconditioner $M_{\mathrm{sD}}$.
The implementation is done using G+Smo~\cite{gismoweb}. The local subproblems
are solved with the sparse direct solver from the PARDISO project\footnote{\url{https://www.pardiso-project.org/}}.

For all numerical experiments, we start the iteration
with a randomly sampled vector with entries in the interval $[-1,1]$.
The stopping criterion is chosen as follows. The iteration stops if the
the $l_2$-norm of the residual vector drops below $10^{-6}$ times the
residual of the right-hand side. For each experiment, we show the number of iterations (it)
required to reach the stopping criterion and an estimate
of the condition number ($\kappa$)
of the preconditioned system matrix $M_{sD}F$
that has been derived using the PCG algorithm.

We start with the numerical experiments for the circular ring (Figure~\ref{subfig:Rings}). Table~\ref{tab:TripleRing_vertex} shows the 
results for Alg.~A (vertex values). We observe in that the condition number grows not more than
$\log^2 H/h$. Moreover, we also observe a weak growth in the spline degree,
which is smaller than the linear growth predicted by the theory.
Table~\ref{tab:TripleRing_edge} shows the results for Alg.~B (edge averages). 
Here, the condition numbers are much larger than before. The dependence on the
grid size and the spline degree seems to be the same as for Alg.~A.
The Table~\ref{tab:TripleRing_vertexEdge} comprises the data collected for Alg.~C (vertex values + edge averages). As expected, this approach yields the smallest
values for the condition numbers. The condition number seems to grow only like 
$\sqrt{p}$ in the spline degree and like $\log H/h$ in the grid size.

Next, we will take a look at the results on the Yeti-footprint (Figure~\ref{subfig:Yeti}). Table~\ref{tab:Yeti_vertex} reports on the results for
Alg.~A. We can again see that the increase in the condition number is like $\log^2 H/h$ and the increase in $p$ is sub-linear, almost like $\sqrt{p}$.
Table~\ref{tab:Yeti_edge} gives the condition number estimates for Alg.~B,
where we observe that the condition number grows almost
linearly in the spline degree $p$ (which indicates that the convergence theory
might be sharp in this respect). 
Alg.~C, whose results are given in Table~\ref{tab:Yeti_vertexEdge}, yields
again the best condition numbers.  Compared
to the circular ring, the condition numbers for the Yeti-footprint are smaller, which
might be connected to a more regular geometry mapping.

Finally, we present an experiment that indicates that the dependence of the condition number on the ratio of the grid sizes of neighboring patches ($\max_{k=1, \dots, K} \max_{\ell \in \mathcal{N}_\Gamma(k)} \frac{h_{k}}{h_{\ell}}$) is only present for Alg.~B. For this purpose, we compare Alg.~A with Alg.~B on both computational domains 
for particular grids that are constructed as follows.
Starting from the initial grid introduced above, we applied $4$ refinement
steps as above. Then, we uniformly refine the grids on all patches $\Omega^{(k)}$, where $k$ is even, additionally $e\in\{1,2,3\}$ times. Thus, the grid sizes between patches with even and odd degrees vary by a factor of $2^e$. 
We observe in Table~\ref{tab:ratio dependence ring} for the circular ring and in Table~\ref{tab:ratio dependence yeti} for the Yeti-footprint that the condition
number is almost independent of $e$ if Alg.~A is used, while it increases
like $2^e$ if Alg.~B is chosen. This means that the condition number grows linearly
in the ratio of the grid sizes, which coincides with the prediction of the
convergence theory.

\section{Conclusions}
\label{sec:6}
In this paper, we have extended the theory from \cite{SchneckenleitnerTakacs:2019}, 
where we established $p$-explicit condition number estimates for continuous
Galerkin IETI-DP solvers, to symmetric interior penalty
discontinuous Galerkin discretizations. Again, we have analyzed both the dependence
on the grid size and the spline degree. If the vertex values are chosen as primal
degrees of freedom (Alg.~A and C), the results are the same as for conforming
discretizations. If we use only the edge averages (Alg.~B), the condition number
estimate additionally depends on the ratio between the grid sizes of neighboring patches.
This can also be observed in the numerical experiments.

Alg.~B does not perform as good as the other options. However, this approach
seems to be beneficial if a non-conforming decomposition of the overall domain
into patches is considered, like a decomposition with T-junctions.
Although the IETI-DP methods also perform well on domains with non-trivial
geometry functions, analyzing the dependence of the condition
number on the geometry function is an interesting topic for future research.

\textbf{Acknowledgments.}
The first author was supported by the Austrian Science Fund (FWF): S117 and 
W1214-04. Also, the second author has received support from the Austrian Science
Fund (FWF): P31048.
Finally, the authors want to thank Ulrich Langer for fruitful discussions and help
with the study of existing literature.

\bibliography{literature.bib}

\end{document}